\documentclass[oneside, 11pt, article, draft]{memoir}

\usepackage{amssymb}
\usepackage{amsmath}
\usepackage{amsfonts}
\usepackage{amsthm}
\usepackage{thmtools}

\usepackage[english]{babel}

\usepackage[final]{graphicx}

\usepackage[plain, english]{fancyref}
\usepackage{hyperref}
\hypersetup{
     colorlinks   = true,
     citecolor    = blue,
     linkcolor    = blue,
     final
}

\usepackage{bbm}
\usepackage{mathrsfs}
\usepackage{stmaryrd}

\usepackage[all, line]{xy}

\usepackage{ifthen}

\usepackage{xspace}

\usepackage{calc}
\usepackage{color}

\def\clap#1{\hbox to 0pt{\hss#1\hss}}

\makeatletter
\def\blfootnote{\xdef\@thefnmark{}\@footnotetext}
\makeatother

\newcommand{\ig}{\includegraphics}

\newcommand{\mathset}[1]{\mathbbm{#1}}

\newcommand{\setZ}{\mathset{Z}}

\newcommand{\setR}{\mathset{R}}
\newcommand{\setC}{\mathset{C}}

\newcommand{\idone}{\mathbbm{1}}

\newcommand{\bwedge}{\raise1pt\hbox{\ensuremath{\bigwedge}}}

\newcommand{\abs}[2][]{#1\lvert #2#1\rvert}

\newcommand{\frakg}{\mathfrak{g}}
\newcommand{\frakh}{\mathfrak{h}}

\newcommand{\frakp}{\mathfrak{p}}

\renewcommand{\phi}{\varphi}
\renewcommand{\epsilon}{\varepsilon}
\renewcommand{\otimes}{\varotimes}

\newcommand{\sotimes}{\mathbin{\!\!\raise1.5pt\hbox{
      $\scriptscriptstyle\otimes$}}}

\newcommand{\rb}{\raisebox}
\newcommand{\rrlen}{-3.75mm}
\newcommand{\rwlen}{9mm}

\newcommand{\ssf}{{\raise0.4mm\hbox{\ensuremath{\scriptscriptstyle \:\!\phi}}}} % small f for framed objects
\newcommand{\knots}{\mathcal{K}} % the space of knots
\newcommand{\fknots}{\mathcal{K}^\ssf} % framed knots
\newcommand{\vassinv}{\mathcal{V}} % Vassiliev invariants
\newcommand{\fvassinv}{\mathcal{V}^\ssf} % Framed Vassiliev invariants
\newcommand{\vassinvc}{\mathcal{\widehat V}} % Completion of \vassinv
\newcommand{\fvassinvc}{\mathcal{\widehat V}^\ssf} % Completion of \fvassinv
\newcommand{\weights}{\mathcal{W}} % Weigh systems
\newcommand{\fweights}{\mathcal{W}^\ssf} % Weigh systems
\newcommand{\fweightsc}{\mathcal{\widehat W}^\ssf} % Weigh systems

\newcommand{\cd}{A} % The set of Chord diagrams
\newcommand{\cda}{\mathcal{A}} % The algebra of Chord diagrams
\newcommand{\fcda}{\mathcal{A}^\ssf} % The algebra of framed Chord diagrams
\newcommand{\fcdac}{\mathcal{\widehat A}^\ssf} %The completion of \fcda 

\newcommand{\qinv}{Q} % Quantum invariant

\newcommand{\sleq}{\raise0.2mm\hbox{\ensuremath{\scale{0.38}{\leqslant}}}}

\newcommand{\ssm}{{\scalebox{0.9}[1.0]{$\:\!\scriptstyle - \!\!\;$}}}

\newcommand{\so}{\mathfrak{so}}
\renewcommand{\sl}{\mathfrak{sl}}

\DeclareMathOperator{\spn}{span}

\DeclareMathOperator{\Id}{Id}

\DeclareMathOperator{\Tr}{Tr}
\DeclareMathOperator{\Hom}{Hom}     
\DeclareMathOperator{\End}{End}
\DeclareMathOperator{\Aut}{Aut}

\theoremstyle{plain}
\newtheorem{theorem}{Theorem}[chapter]

\newtheorem{proposition}[theorem]{Proposition}

\theoremstyle{definition}
\newtheorem{remark}[theorem]{Remark}
\newtheorem{definition}[theorem]{Definition}

\declaretheorem[style=definition,qed=$\diamond$,
sibling=definition]{example}  

\makeatletter       % Gør at vi kan bruge @ i kommandonavne.

\let\@xp\expandafter % Vi skal bruge \expandafter en del, så en
\newcommand\DefineFancyrefPrefix[2]{%
  \@namedef{fancyref#1labelprefix}{#1}%
  \@namedef{Fref#1name}{#2}%
  \@namedef{fref#1name}{\MakeLowerCase{\@nameuse{Fref#1name}}}%
  \def\@style{vario}%
  \@xp\@xp\@xp\frefformat\@xp\@xp\@xp\@style\@xp\csname
  fancyref#1labelprefix\endcsname
  {%
    \@nameuse{fref#1name}\fancyrefdefaultspacing##1##3%
  }%
  \@xp\@xp\@xp\Frefformat\@xp\@xp\@xp\@style\@xp\csname
  fancyref#1labelprefix\endcsname
  {%
    \@nameuse{Fref#1name}\fancyrefdefaultspacing##1##3%
  }%
  \def\@style{plain}%
  \@xp\@xp\@xp\frefformat\@xp\@xp\@xp\@style\@xp\csname
  fancyref#1labelprefix\endcsname
  {%
    \@nameuse{fref#1name}\fancyrefdefaultspacing##1%
  }%
  \@xp\@xp\@xp\Frefformat\@xp\@xp\@xp\@style\@xp\csname
  fancyref#1labelprefix\endcsname
  {%
    \@nameuse{Fref#1name}\fancyrefdefaultspacing##1%
  }%
}
\makeatother

\DefineFancyrefPrefix{ax}{Axiom}
\DefineFancyrefPrefix{cor}{Corollary}
\DefineFancyrefPrefix{lem}{Lemma}
\DefineFancyrefPrefix{prop}{Proposition}
\DefineFancyrefPrefix{thm}{Theorem}
\DefineFancyrefPrefix{def}{Definition}
\DefineFancyrefPrefix{rem}{Remark}
\DefineFancyrefPrefix{exa}{Example}
\DefineFancyrefPrefix{cha}{Chapter}

\makepagestyle{niels}

\makeoddfoot{niels}{}{\thepage}{}
\setlrmarginsandblock{3cm}{3cm}{*} % left and right 
\setulmarginsandblock{3.5cm}{3.5cm}{*} % upper and lower
\checkandfixthelayout[nearest]

\chapterstyle{article}
\AtEndDocument{\bigskip{\footnotesize
    
    \hfill
    \begin{minipage}{\dimexpr\textwidth-1cm}
      \vspace{0.5cm}
      (I.~Biswas) \textsc{School of Mathematics, Tata Institute of
        Fundamental Research,\\ Homi Bhabha Road, Bombay 400005,
        India}\par
      \textit{E-mail address: }\texttt{indranil@math.tifr.res.in}\par
      \addvspace{\medskipamount}

      (N.L.~Gammelgaard) \textsc{Centre for Quantum Geometry of Moduli
        Spaces (QGM), \\Aarhus University, Ny Munkegade 118, bldg. 1530,
        DK-8000 Aarhus C, Denmark}\par
      \textit{E-mail address: }\texttt{nlg@qgm.au.dk}\par
      \addvspace{\medskipamount}

      \xdef\tpd{\the\prevdepth}
    \end{minipage}

}}

\begin{document}

\title{Vassiliev Invariants from Symmetric Spaces}

\author{Indranil Biswas and Niels Leth Gammelgaard}

\blfootnote{This work was supported by a Marie Curie International
  Research Staff Exchange Scheme Fellowship within the 7th European
  Union Framework Programme (FP7/2007-2013) under grant agreement
  no. 612534, project MODULI - Indo European Collaboration. The second
  author was partly supported by the center of excellence grant
  `Center for Quantum Geometry of Moduli Spaces' from the Danish
  National Research Foundation (DNRF95). The first author is supported
  by the J. C. Bose Fellowship.  }

\date{\vspace{-5ex}}

\maketitle

\begin{abstract}
  We construct a natural framed weight system on chord diagrams from
  the curvature tensor of any pseudo-Riemannian symmetric space. These weight
  systems are of Lie algebra type and realized by the action of the
  holonomy Lie algebra on a tangent space. Among the Lie
  algebra weight systems, they are exactly characterized by having the
  symmetries of the Riemann curvature tensor.
\end{abstract}

% \listoffixmes

\chapter{Introduction}

The essence of this paper is the simple observation that the curvature
tensor of a pseudo-Riemannian symmetric space satisfies an identity analogous
to the 4T relation in the theory of Vassiliev invariants. This means
that the curvature tensor gives rise to a weight system on chord
diagrams in the most natural way, and by virtue of the Kontsevich
integral, this weight system integrates to a finite-type invariant.

Suppose that $(M,g)$ is a connected pseudo-Riemannian manifold, with
Levi-Civita connection $\nabla$ and curvature $R \in C^\infty(M, T^*\!M
\otimes T^*\!M \otimes \End(TM))$. Using the metric to identify $T^*\!M$ with
$TM$, we get a tensor $\hat R \in C^\infty(M, \End(TM) \otimes \End(TM))$. This
can be used to construct a function $w_R$ on chord diagrams, by
placing $\hat R$ on each chord and contracting around the circle, as
done in the following example
\begin{align*}
  w_R (\rb{\rrlen}{\ig[width=\rwlen*105/100]{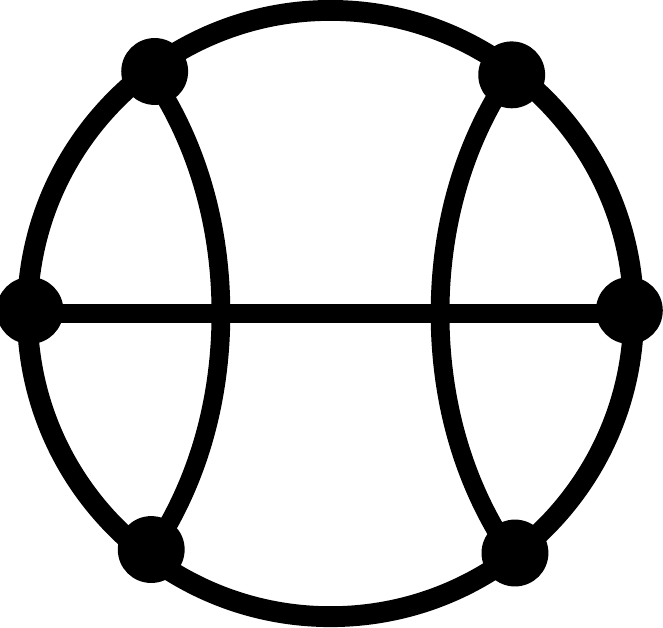}}) 
  \quad  = \quad
  \rb{-9.9mm}{
    \begin{picture}(52,62)(0,0)
      \put(0,0){\ig[height=22mm]{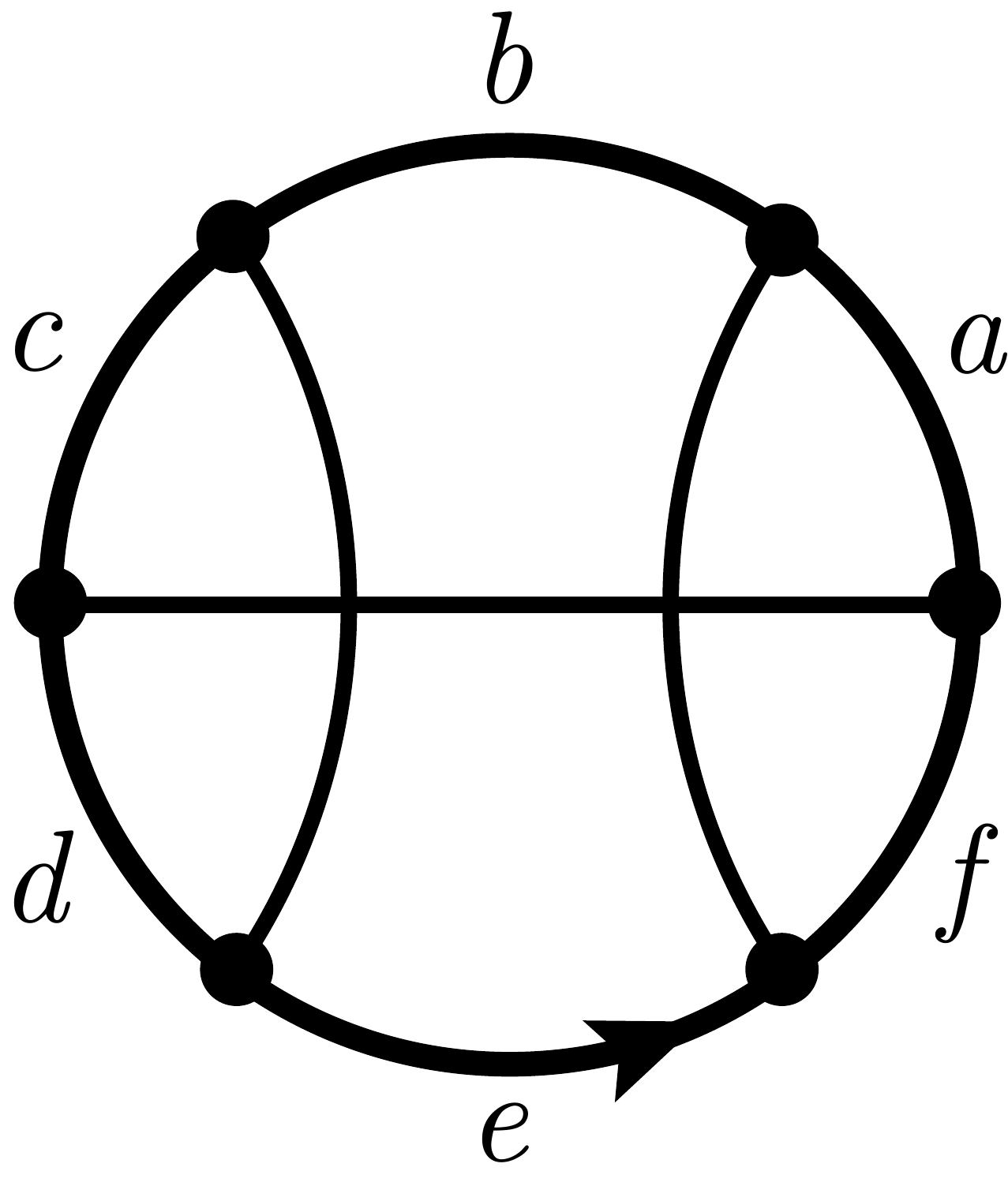}} 
      \put(23,33){\mbox{$\scriptstyle \hat R$}}
      \put(38,37){\mbox{$\scriptstyle \hat R$}}
      \put(8,21){\mbox{$\scriptstyle \hat R$}}
    \end{picture}
  } 
  \quad = \quad
  \hat R^{bf}_{ae} \hat R^{ad}_{cf} \hat R^{ce}_{bd} \quad \in \quad C^\infty(M).
\end{align*}
If $(M,g)$ is a symmetric space, the curvature tensor is
parallel. This will not only have the obvious implication that
$\omega_R$ produces constant functions on $M$, but also that it
satisfies the $4T$ relation on chord diagrams. Therefore, it
defines a weight system and a finite-type invariant of knots through
the universal Vassiliev invariant.

The algebra of chord diagrams, with the $4T$ relation, can be
equivalently represented as closed trivalent Jacobi diagrams, modulo
the IHX and AS relations. Through this viewpoint, there is a striking
connection to metrized Lie algebras, which can be used to construct
weight systems \cite{MR1318886}. This is done by placing the totally
antisymmetric structure tensor $Y \in \frakg^{\otimes 3}$ at each
trivalent internal vertex and contracting edges using the metric. The
Jacobi identity ensures that the IHX relation is satisfied, and one
obtains a central element in the universal enveloping algebra. By
applying a representation of the Lie algebra and taking the trace,
this construction produces a numerical weight system. As we shall see,
the weight systems defined by symmetric spaces are in fact of this Lie
algebra type and realized through the holonomy Lie algebra and its
action on the tangent spaces of the symmetric space.

Using the curvature of a hyper-K\"ahler metric to construct weight
systems was previously proposed by Rozansky and Witten
\cite{MR1481135}. Soon after, it was realized by Kapranov and
Kontsevich \cite{MR1671737,MR1671725} that the construction would also
work in the holomorphic symplectic setting, using the Atiyah class in
place of the curvature. At each trivalent vertex of a Jacobi diagram,
Rozansky and Witten put a copy of the curvature, viewed as one-form $R
\in \Omega^{0,1}(M, T^* \otimes T^* \otimes T^*)$, and contracted
using the holomorphic symplectic form. The one-form part of $R$ is
left out of this contraction, so the result gives a cohomology class
of degree equal to the number of trivalent vertices. It can be
integrated if the dimension of the manifold is appropriate.

The Bianchi identity plays an important role, analogous to the Jacobi
identity, in proving the IHX relation for Rozansky and Witten.  In our
case, the Bianchi identity plays a central role in proving that the
weight systems coming from a symmetric space are of Lie algebra
type. It literally ensures the Jacobi identity for the symmetric
triple (\Fref{def:6}) of the symmetric space. On the other hand, it
is not needed to see that symmetric spaces give rise to weight systems
in the first place. This underlines the curious fact, that weight
systems coming from symmetric spaces are most naturally expressed on
chord diagrams, whereas the Rozansky-Witten weight systems and
Bar-Natan's Lie algebra weights seem to favor the trivalent Jacobi diagrams.

The paper is organized as follows. In Section
\ref{cha:vassiliev-invariants}, we recall the basic theory of
Vassiliev invariants, establishing notation and introducing the Jones
and Yamada polynomials, which will be used to illustrate
throughout. The definitive resources on this theory are Bar-Natan's
paper \cite{MR1318886} and the excellent book \cite{MR2962302}.  In
Section \ref{cha:constr-weight-syst}, we give a simple framework for
constructing weight systems, using what we call a weight tensor, which
must satisfy a variant of the 4T relation. The ubiquitous weight
systems coming from representations of a metrized Lie algebras fit
this description, and the weight tensor is obtained by applying the
representation to the Casimir tensor. Then we prove the main
observation of the paper, which is that the curvature of a
pseudo-Riemannian symmetric space defines a weight tensor.  Finally,
in Section \ref{cha:symmetric-spaces-lie}, we prove that the weight
systems obtained from symmetric spaces are of Lie algebra type and
realized through the holonomy Lie algebra. Moreover, we use the
correspondence between pseudo-Riemannian symmetric spaces and
symmetric Lie algebra triples \cite{MR556610} to show that a weight
tensor of Lie algebra type can be realized on a symmetric space
exactly if it carries the symmetries of a curvature tensor.

\pagebreak[2]
\chapter{Vassiliev Invariants}
\label{cha:vassiliev-invariants}

A knot invariant is a function on the space $\knots$ of ambient
isotopy classes of oriented knots in $\setR^3$. Among such
functions, the \emph{Vassiliev} or \emph{finite-type} invariants
correspond in a certain sense to polynomial functions. To explicate
the polynomial nature of these invariants, one usually considers the
larger space of singular knots, which are allowed to have double
points.
\begin{center}
  \label{examples-knots}
  \begin{tabular}{c@{\qquad\qquad}c}
    \ig[height=16mm]{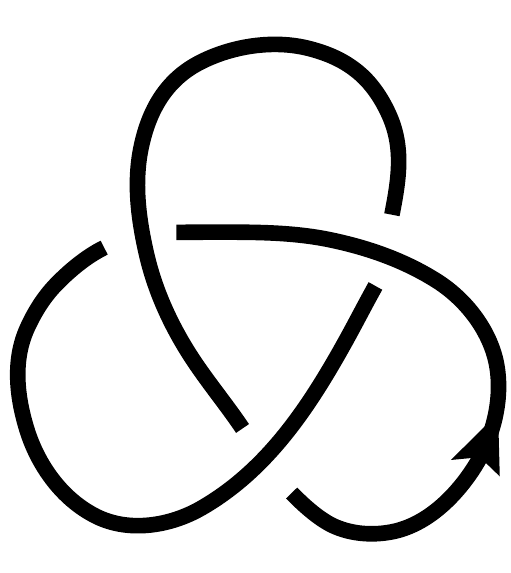} & \ig[height=16mm]{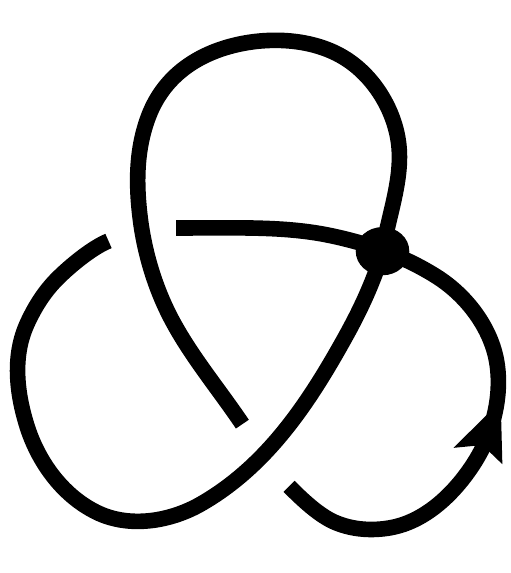}
    \\
    \scriptsize A knot & \scriptsize A singular knot
  \end{tabular}
\end{center}

The set of singular knots with $n$ singular points is denoted by
$\knots_n$.
Any knot invariant $v$ can be extended to singular knots through the
skein relation
\begin{align}
  \label{eq:1}
  v(\rb{\rrlen}{\ig[width=\rwlen]{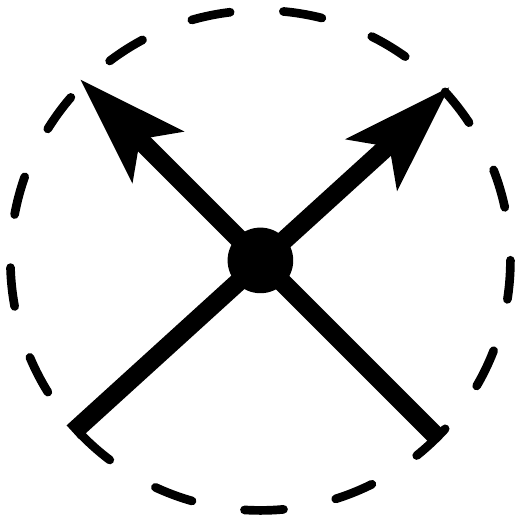}}) = 
  v(\rb{\rrlen}{\ig[width=\rwlen]{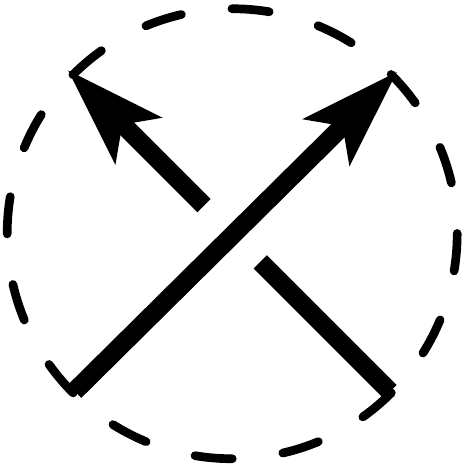}}) - 
  v(\rb{\rrlen}{\ig[width=\rwlen]{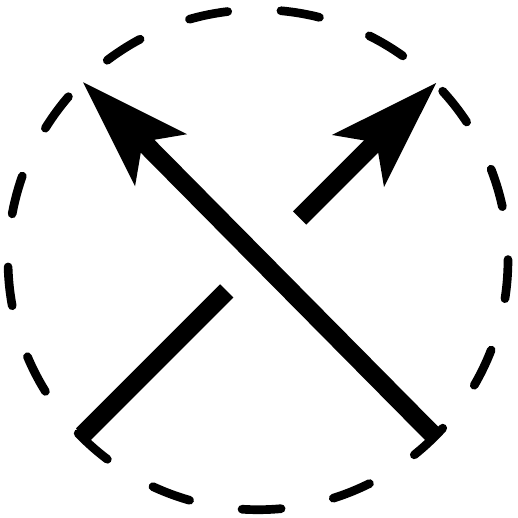}}).
\end{align}
A \emph{Vassiliev} invariant of order at most $n$ is a knot invariant
whose extension vanishes on any singular knot with at least $n+1$
singularities. The extension of a knot invariant to singular knots,
through the equation \eqref{eq:1}, can then be viewed as a derivative of
the invariant. The restriction to knots with $n$ double points
corresponds to the $n$'th derivative, so Vassiliev invariants of order
$n$ are exactly those with vanishing derivative of order
$n+1$. In this sense they are analogous to polynomials of order $n$.

The space of Vassiliev invariants of order $n$ is denoted by
$\vassinv_n$. Clearly we have a natural filtration,
\begin{align*}
  \vassinv_0 \subset \vassinv_1 \subset \vassinv_{2} \subset \ldots
  \subset \bigcup_{n\geq 0} \vassinv_n = \vassinv.
\end{align*}

\begin{example}
  \label{exa:2}
  The Jones polynomial is an invariant of oriented links, with values
  in the Laurent polynomials $\setZ[t^{\frac{1}{2}},t^{\ssm
    \frac{1}{2}}]$. It can be defined by the skein relations
  \begin{align*}
    t^{\ssm 1} J(\rb{\rrlen}{\ig[width=\rwlen]{fig/pcross}}) -
    t \:\!J(\rb{\rrlen}{\ig[width=\rwlen]{fig/ncross}}) = (t^{\frac{1}{2}} -
    t^{\ssm {\frac{1}{2}}}) \:\!
    J(\rb{\rrlen}{\ig[width=\rwlen]{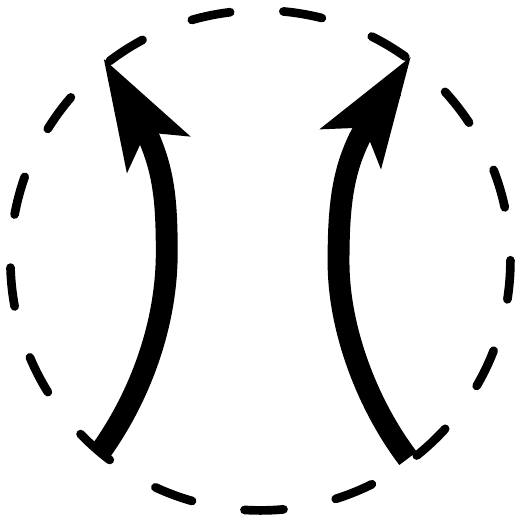}}) \quad \qquad
    \text{and} \quad \qquad
    J(\rb{\rrlen}{\ig[width=\rwlen]{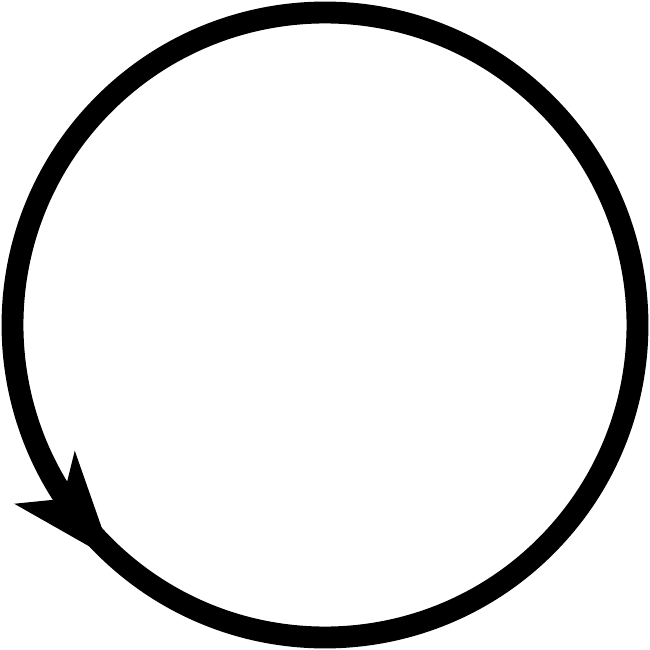}}) = 1.
  \end{align*}
  It follows easily that the Jones polynomial of a knot actually takes
  values in $\setZ[t,t^{\ssm 1}]$. After substituting $t = e^h$ and
  expanding the Jones polynomial as a power series in $h$, the
  coefficient $j_n$ of $h^n$ turns out to be a Vassiliev invariant of
  order $n$ \cite{MR1198809,MR1318886}.  Indeed, it is easy to check
  that
  \begin{align*}
    J(\rb{\rrlen}{\ig[width=\rwlen]{fig/double}}) = 
    J(\rb{\rrlen}{\ig[width=\rwlen]{fig/pcross}}) -
    J(\rb{\rrlen}{\ig[width=\rwlen]{fig/ncross}}) = h ( \,\cdots ),
  \end{align*}
  so the Jones polynomial of a knot $K$ with $n+1$ singularities will
  be divisible by $h^{n+1}$, and in particular $j_n(K)$ will
  vanish. The invariants $j_n$ are called the Jones invariants.
\end{example}

The Jones polynomial is of course not a Vassiliev itself, but we can
view it as a formal power series with coefficients in Vassiliev
invariants. To capture such invariants, one introduces the space of
\emph{power series} Vassiliev invariants
\begin{align*}
  \vassinvc := \prod_{n\geq 0} \vassinv_n \, .
\end{align*}

The natural pointwise multiplication of invariants turns the space
$\vassinv$ of all Vassiliev invariants into a filtered algebra. In
fact, this space also has a natural coproduct $\Delta \colon \vassinv
\to \vassinv \otimes \vassinv$, dual to the operation of connect sum of knots,
\begin{align*}
  \Delta(v)(K_1,K_2) = v(K_1 \# K_2).
\end{align*}
This gives $\vassinv$ the structure a filtered bialgebra.

To any singular knot, we can assign a chord diagram on the circle,
encoding the order of singular points along the knot.  With $A_n$
being the set of chord diagrams with $n$ chords, this assignment
defines a map

\begin{center}
  $\delta \colon \knots_n \to A_n$ \hspace{1cm} e.g. \hspace{1cm}
  \rb{-6.2mm}{\ig[height=16mm]{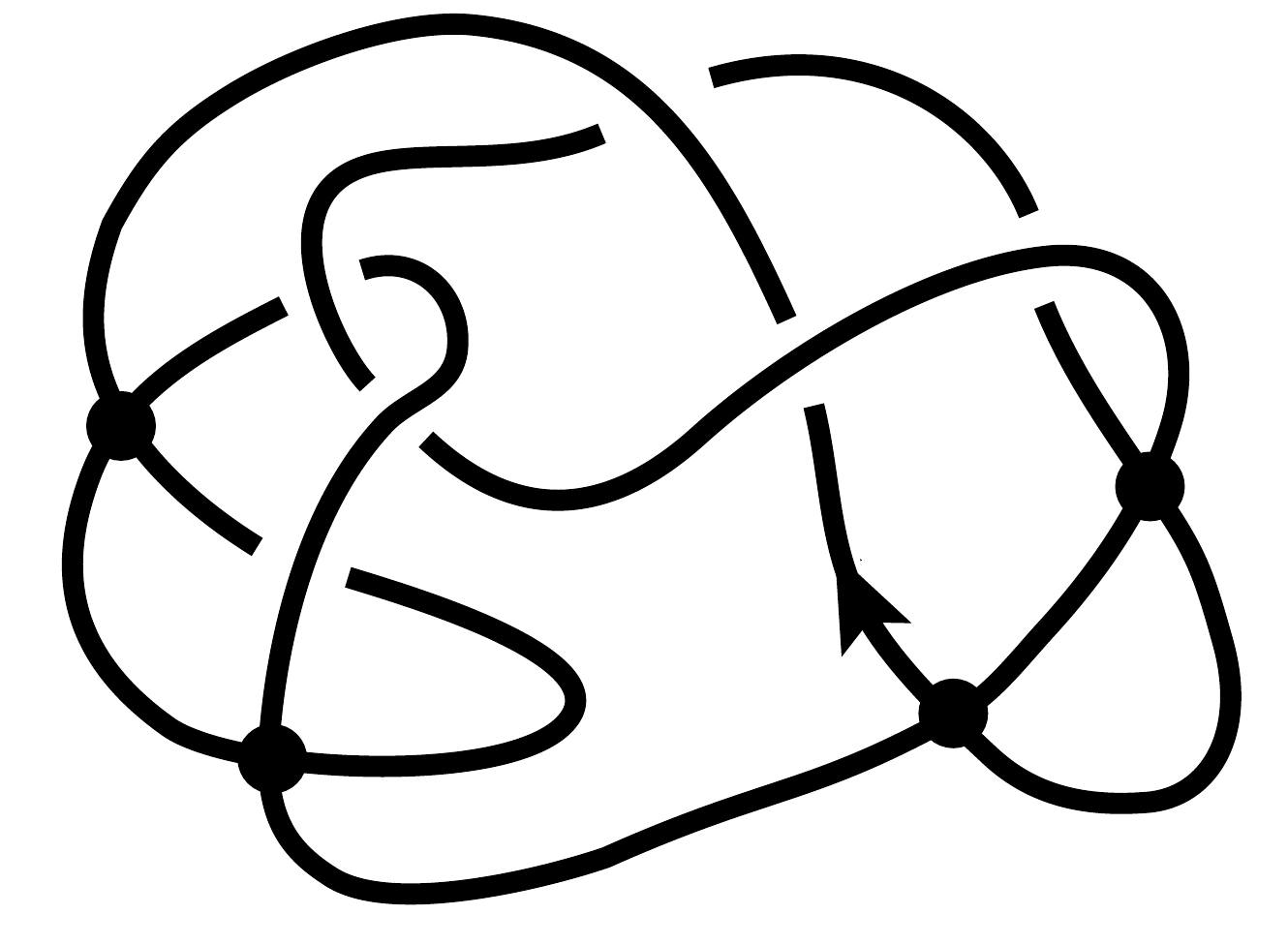}} \hspace{0.1cm} $
  \longmapsto $ \hspace{0.1cm}
  \rb{-4.9mm}{\ig[height=12mm]{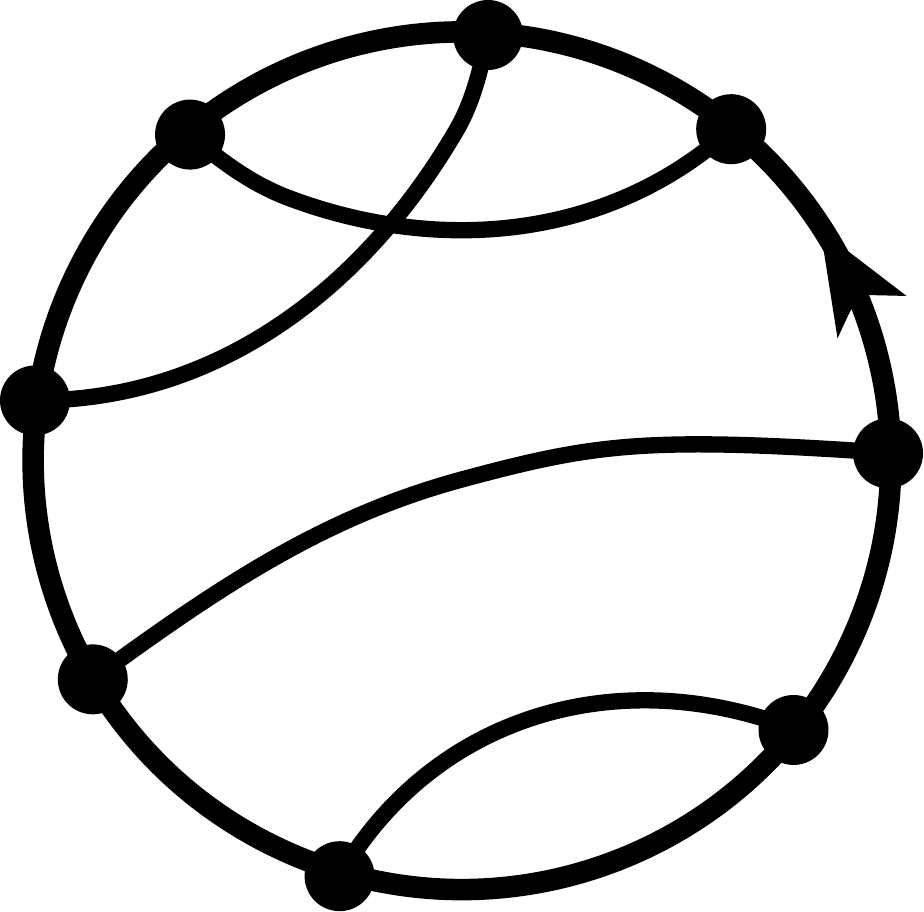}}.
\end{center}
The circle represents the knot, and chords connect points corresponding
to the same double point. If $K \in \knots_n$ is a knot with $n$
singularities and $v \in \vassinv_n$ is a Vassiliev invariant of order
at most $n$, then $v(K)$ is invariant under crossing changes on
$K$. In other words, the evaluation $v(K)$ only depends on the order of
the singularities along the knot, or simply its associated chord
diagram. This means that $v$ induces a map, called the \emph{symbol}
of $v$,
\begin{align*}
  \sigma_n(v) \colon \cd_n \to \setC, \qquad \qquad \sigma_n(v)(D) =
  v(K_{\!\!\:D}),
\end{align*}
where $K_D \in \knots_n$ is any knot with $\delta(K_D) = D$.

Not every map on chord diagrams can be realized as the symbol of a
Vassiliev invariant. Indeed, it is not difficult to show that such a
symbol must satisfy the conditions of the following definition.

\begin{definition}
  \label{def:1}
  Any map $w \colon \cd_n \to \setC$ satisfying the 4T relation,
  \begin{align}
    \label{eq:4Tw}
    w \bigl(\rb{\rrlen}{\ig[width=\rwlen]{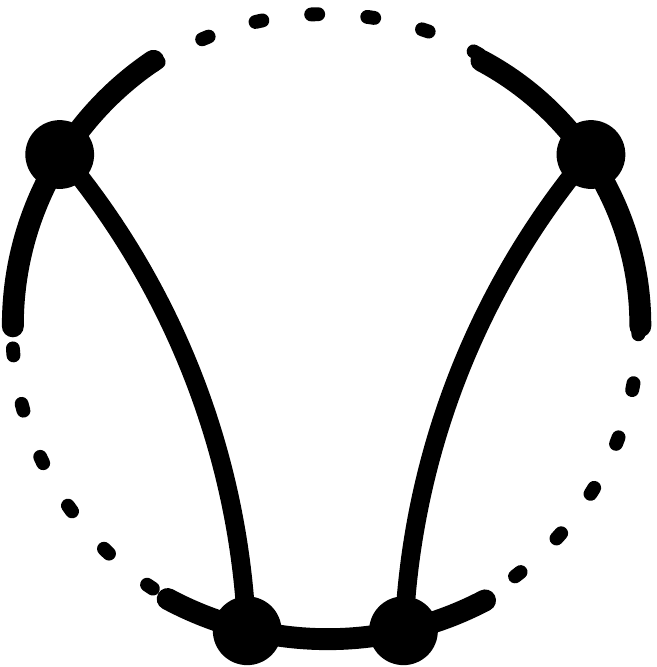}} \bigr) - w
    \bigl(\rb{\rrlen}{\ig[width=\rwlen]{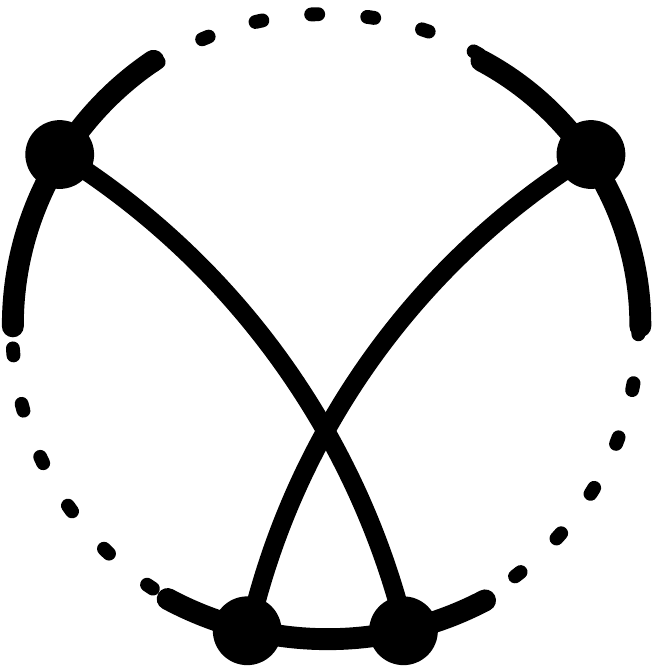}} \bigr) + w
    \bigl(\rb{\rrlen}{\ig[width=\rwlen]{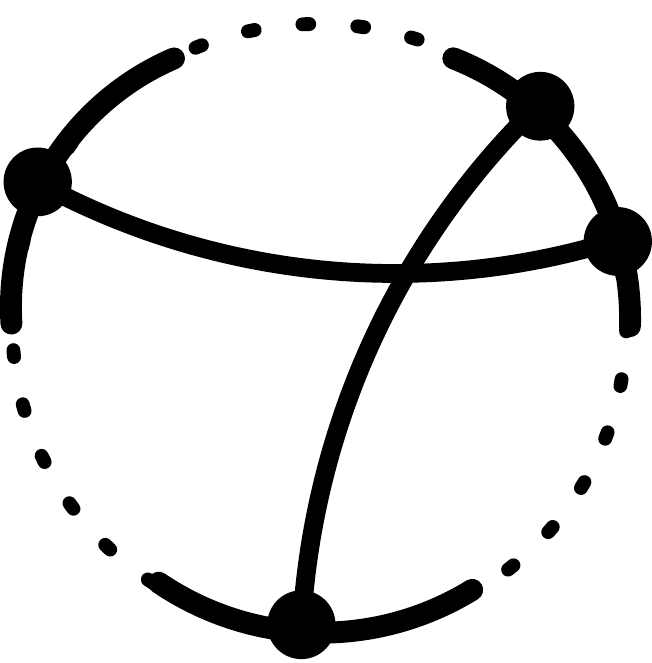}} \bigr) - w
    \bigl(\rb{\rrlen}{\ig[width=\rwlen]{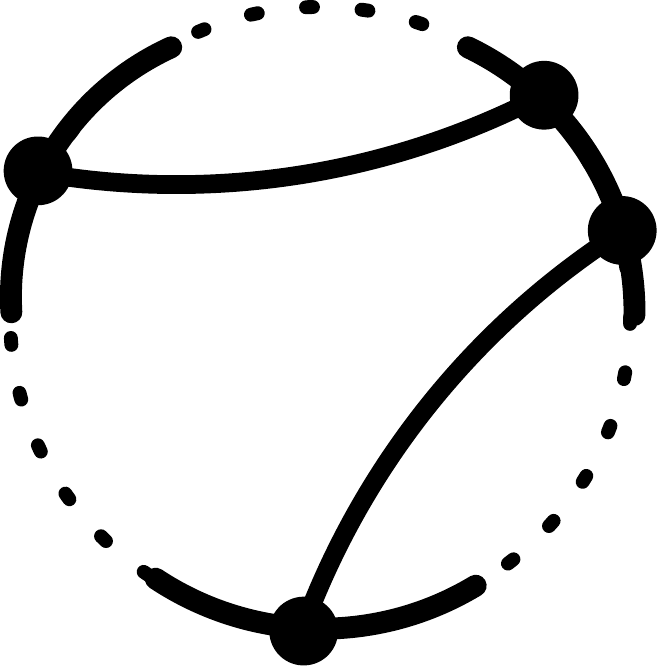}} \bigr) = 0,
  \end{align}
  is called a (framed) \emph{weight system} of order $n$. If in
  addition it satisfies the 1T relation,
  \begin{align}
    \label{eq:1Tw}
    w(\rb{\rrlen}{\ig[width=\rwlen]{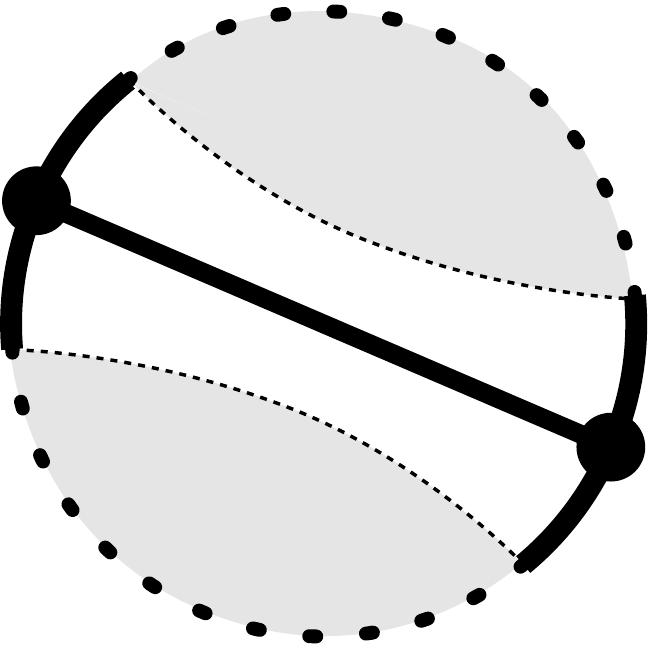}}) = 0,
  \end{align}
  then it is called an \emph{unframed} weight system.
\end{definition}

The diagrams in the 4T relation can have other chords with endpoints
on the dotted parts of the circle and possibly intersecting the shown
chords. The same holds for the 1T relation, except the other chords
are not allowed to intersect the shown chord, meaning they must stay in
the gray regions.

The vector space of (unframed) weight systems of order $n$ is denoted
by $\weights_n$. Since the symbol of a Vassiliev invariant defines a
weight system, we see that the symbol defines a map $\sigma_n \colon
\vassinv_n \to \weights_n$. Clearly two Vassiliev invariants of order
$n$ have the same symbol if and only if their difference is in
$\vassinv_{n-1}$, which is therefore the kernel of $\sigma_n$. This
means that the symbol descends to an injective map
\begin{align*}
  \overline \sigma_n \colon \vassinv_n / \vassinv_{n-1} \to
  \weights_n,
\end{align*}
which is in fact an isomorphism. Indeed, surjectivity follows
immediately from the following fundamental theorem of Kontsevich
\cite{MR1237836}.

\begin{theorem}
  \label{thm:1}
  There exists a map $J_n \colon \weights_n \to \vassinv_v$ such
  that $\sigma_n \circ J_n = \Id$.
\end{theorem}

The proof uses the celebrated Kontsevich integral to construct the
invariants on Morse knots. There is also a combinatorial construction
using the Drinfeld associator \cite{MR1221650,MR1324388}. The
resulting map $J_n$ yields a preferred Vassiliev invariant having a
given weight system as its symbol, and such invariants in the image of
$J_n$ are called \emph{canonical}. In fact, the Jones invariants
described in \Fref{exa:2} are canonical. This follows by their
relation to quantum invariants, to which we shall return, and the work
of Le-Murakami and Kassel \cite{MR1394520,MR1321145}. The theorem
above is often formulated in terms of the so-called universal
Vassiliev invariant. We shall review this description after briefly
recalling the situation for framed knots.

\subsection{Framed Knots}
The theory of Vassiliev invariants also applies to framed knots, and
in some sense more naturally. In this case, the framing for a
singular knot is allowed to have simple zeros away from the double
points.

\begin{center}
  \label{examples-framed-knots}
  \begin{tabular}{c@{\qquad\qquad}c}
    \ig[height=16mm]{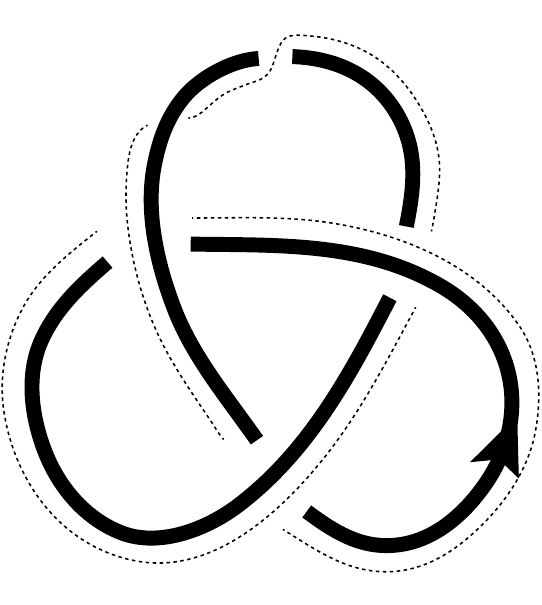} & \ig[height=16mm]{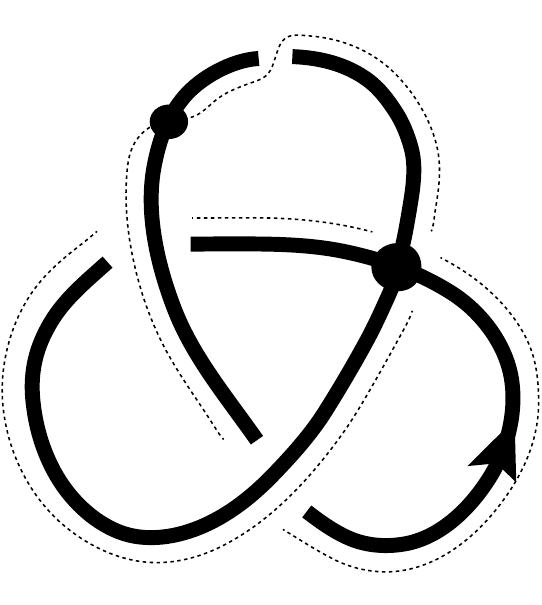}  
    \\
    \scriptsize A framed knot & \scriptsize \clap{A singular framed knot}
  \end{tabular}
\end{center}
The set of singular knots with $n$ singular points, of the knot or the
framing, is denoted by $\fknots_n$. Framed knot invariants are
extended to singular knots by resolving singularities of the framing
through the relation 
\begin{align}
  \label{eq:2}
  v \bigl(\, \rb{-1.7mm}{\ig[width=3.6mm]{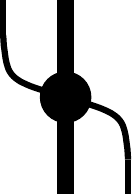}} \,\bigr) =
  v \bigl(\, \rb{-1.7mm}{\ig[width=3.6mm]{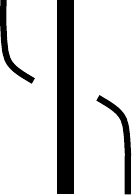}}\, \bigr) -
  v \bigl(\, \rb{-1.7mm}{\ig[width=3.6mm]{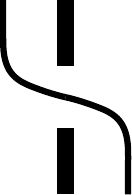}} \,\bigr).
\end{align}
A framed Vassiliev invariant of order $n$ vanishes on all knots with
at least $n+1$ singularities. The space of such invariants is denoted
by $\fvassinv_n$.

\begin{example}
  \label{exa:6}
  The Yamada polynomial \cite{MR1016274} is an invariant of framed
  unoriented links, with values in $\setZ[q^{\frac{1}{2}},q^{\ssm
    \frac{1}{2}}]$. It is defined by the skein relations
  \begin{align*}
    &Y(\rb{\rrlen}{\ig[width=\rwlen]{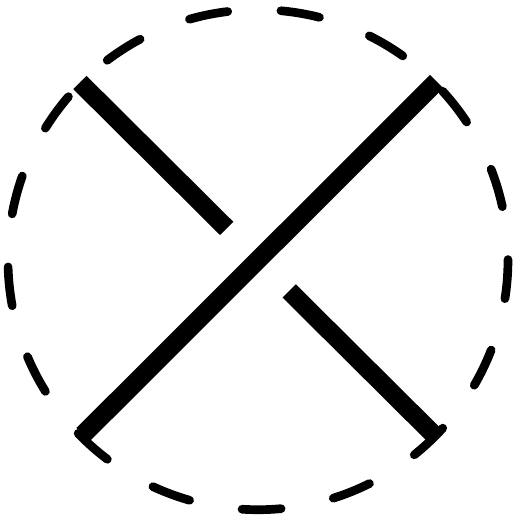}}) - 
    Y(\rb{\rrlen}{\ig[width=\rwlen]{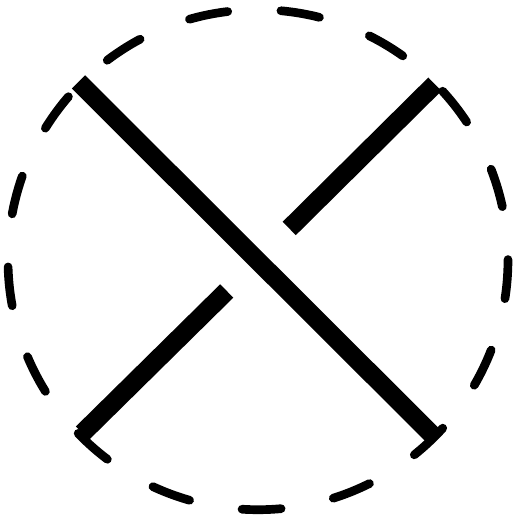}}) \, =\,
    (q^{\frac{1}{2}}-q^{\ssm \frac{1}{2}}) \bigg (
    Y(\rb{\rrlen}{\ig[width=\rwlen]{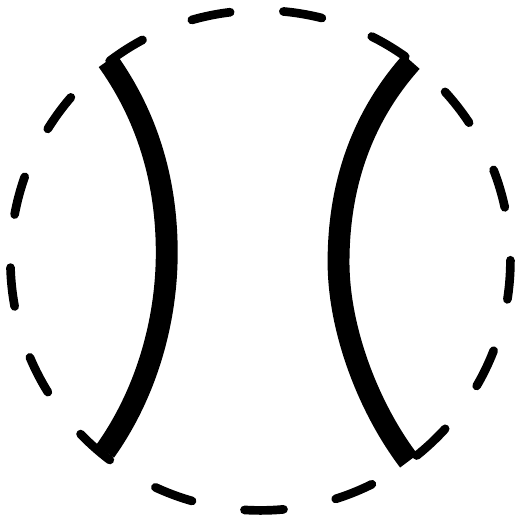}}) -
    Y(\rb{\rrlen}{\ig[width=\rwlen]{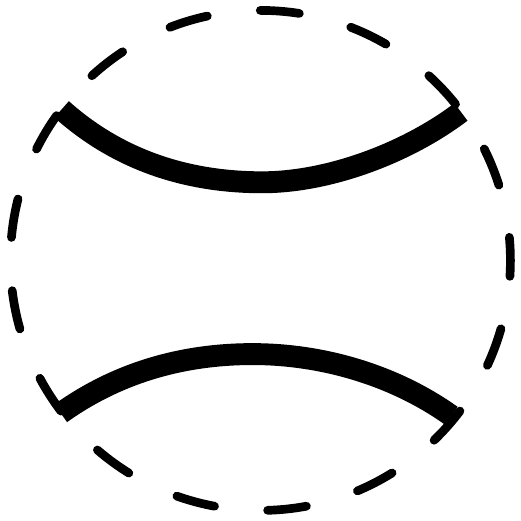}}) \bigg )
  \end{align*}
  with framing and initial conditions
  \begin{align*}
    &q^{\ssm 1} \,
    Y(\rb{\rrlen}{\ig[width=\rwlen]{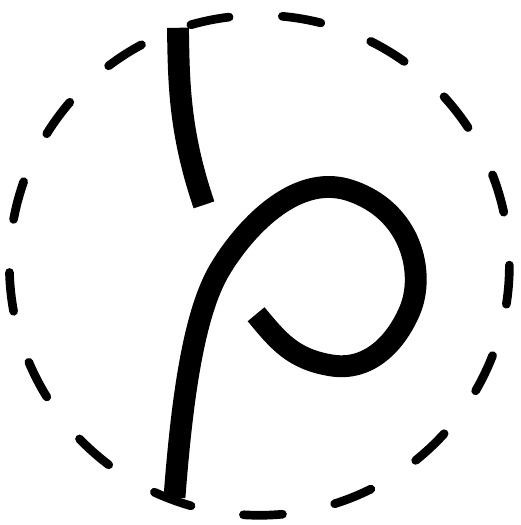}}) \, = \,
    q\, Y(\rb{\rrlen}{\ig[width=\rwlen]{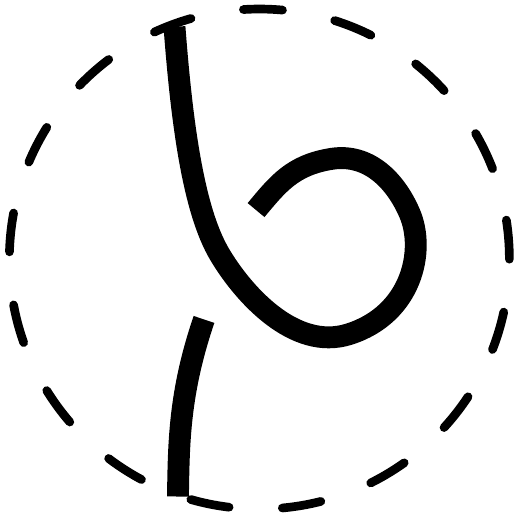}}) \,=\,
    Y(\rb{\rrlen}{\ig[width=\rwlen]{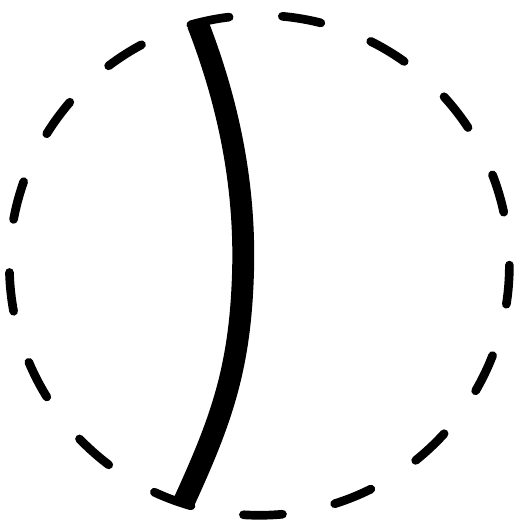}}) \qquad
    \text{and} \qquad
    Y(\rb{\rrlen}{\ig[width=\rwlen]{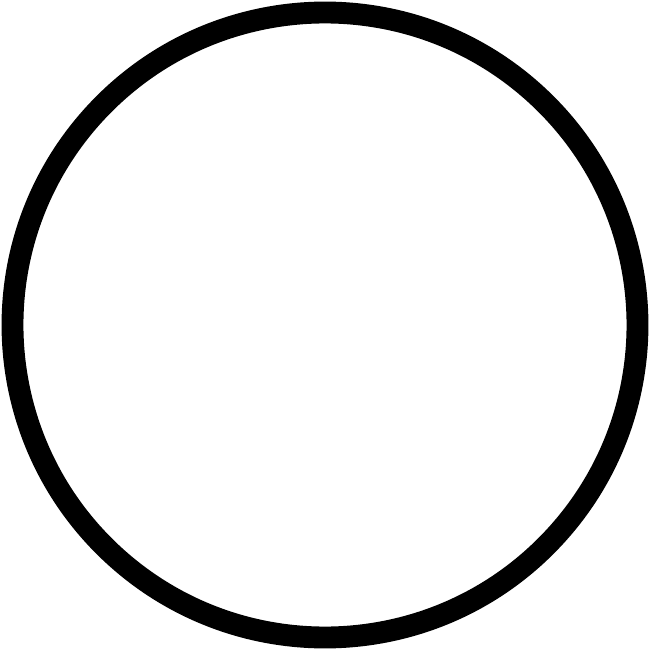}}) \,=\,
    q^{\frac{1}{2}}+q^{\ssm \frac{1}{2}} + 1.
  \end{align*}
  This is a specialization of the two-variable Dubrovnik polynomial,
  which is a variant of the Kauffman polynomial. Of course the Yamada
  polynomial defines an invariant of oriented knots simply by ignoring
  the orientation. In fact, the Yamada polynomial is essentially given
  by evaluating the Jones polynomial on the (2,0)-cabling of a knot
  \nolinebreak[2]\cite{MR1016274}.

  Performing the substitution $q = e^h$ and expanding in powers of
  $h$, the coefficient $y_n$ of $h^n$ will be a Vassiliev invariant of
  order $n$. Indeed, one checks that
  \begin{align*}
    Y(\rb{\rrlen}{\ig[width=\rwlen]{fig/double}}) = 
    Y(\rb{\rrlen}{\ig[width=\rwlen]{fig/pcross-uno}}) -
    Y(\rb{\rrlen}{\ig[width=\rwlen]{fig/ncross-uno}}) = h ( \,\cdots ),
  \end{align*}
  so the Yamada polynomial of a knot $K$ with $n+1$ double points will
  be divisible by $h^{n+1}$, and in particular $y_n(K)$ will vanish.
\end{example}

When calculating the value of a
Vassiliev invariant $v \in \fvassinv_n$ on a framed knot $K \in
\fknots_n$, the zeros of the framing can be replaced by double points
of the knot, simply because 
\begin{align*}
  v(\rb{-1.2mm}{\ig[width=8mm]{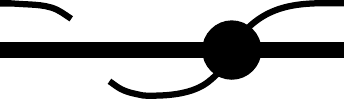}}) = \frac{1}{2}
  v(\rb{-1.55mm}{\ig[width=13mm]{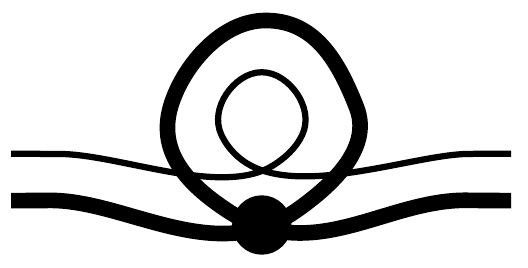}}) - \frac{1}{2}
  v(\rb{-1.2mm}{\ig[width=13mm]{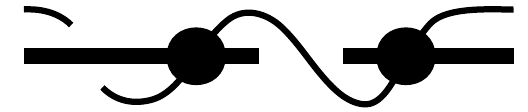}}) = \frac{1}{2}
  v(\rb{-1.55mm}{\ig[width=13mm]{fig/f3}}).
\end{align*}
This means that we can use the same chord diagrams in the framed case.
The map $\delta^\ssf \colon \fknots_n \to A_n$ sends a zero of the
framing to a chord connecting two adjacent points on the circle. As in
the unframed case, we therefore get a symbol map
\begin{align*}
  \sigma^\ssf_n(v) \colon \cd_n \to \setC, \qquad \qquad \sigma^\ssf_n(v)(D)
  = v(K_{\!\!\:D}),
\end{align*}
where $K_D \in \fknots_n$ is any framed knot with $\delta^\ssf(K_D) = D$.

\begin{example}
  \label{exa:8}
  Let us calculate the symbol of the Yamada polynomial.  First of all,
  we observe that
  \begin{align*}
    Y(\rb{\rrlen}{\ig[width=\rwlen]{fig/double}}) = 
    Y(\rb{\rrlen}{\ig[width=\rwlen]{fig/pcross-uno}}) -
    Y(\rb{\rrlen}{\ig[width=\rwlen]{fig/ncross-uno}}) =  h \bigg (
    y_0(\rb{\rrlen}{\ig[width=\rwlen]{fig/twovertical-uno}}) - 
    y_0(\rb{\rrlen}{\ig[width=\rwlen]{fig/twohorizontal-uno}})
    \bigg) + O(h^2).
  \end{align*}
  This means that $y_n$ evaluated on a link with $n$ double points is
  given by the signed sum of $y_0$ evaluated on all possible ways of
  smoothing the double points vertically or horizontally. To write
  down a formula, let $s$ denote a map from the double points of a
  given link $L$ to the set $\{1,-1\}$,\pagebreak[1] and let $L_s$ be
  the link obtained by smoothing each double point of $L$ according to
  the rule
  \begin{align*}
    \rb{\rrlen}{
      \begin{picture}(27,27)(0,0)
        \put(0,0){\ig[width=\rwlen]{fig/double}} 
        \put(10,3){\mbox{$\scriptstyle d$}}
      \end{picture}
    }    
    \rightsquigarrow \rb{\rrlen}{\ig[width=\rwlen]{fig/twovertical-uno}}
    \quad \text{if} \quad
    s(d) = 1 
    \qquad \text{and} \qquad 
    \rb{\rrlen}{
      \begin{picture}(27,27)(0,0)
        \put(0,0){\ig[width=\rwlen]{fig/double}} 
        \put(10,3){\mbox{$\scriptstyle d$}}
      \end{picture}
    }    
    \rightsquigarrow \rb{\rrlen}{\ig[width=\rwlen]{fig/twohorizontal-uno}}
    \quad \text{if} \quad
    s(d) = 1.
  \end{align*}
  If $(-1)^{\abs{s}}$ denotes the sign of $s$, defined as the
  product of its values on all double points, then we have 
  \begin{align*}
    y_n(L) = \sum_s (-1)^{\abs{s}} y_0(L_s) = \sum_s (-1)^{\abs{s}} 3^{c(L_s)},
  \end{align*}
  where the sum is over all maps $s$, and $c(L_s)$ denotes the number
  of components of the smoothing $L_s$. The fact that $y_0(L) = c(L)$
  for any link $L$ is easily verified.
  
  On the level of chord diagrams, this can be visualized in the
  following way. Once again, let $s$ denote a map from the chords of a
  given diagram $D$ to the set $\{1,-1\}$, and let $D_s$ be the
  collection of circles obtained by smoothing each chord according to
  the rule
  \begin{align*}
    \rb{\rrlen}{
      \begin{picture}(27,27)(0,0)
        \put(0,0){\ig[width=\rwlen]{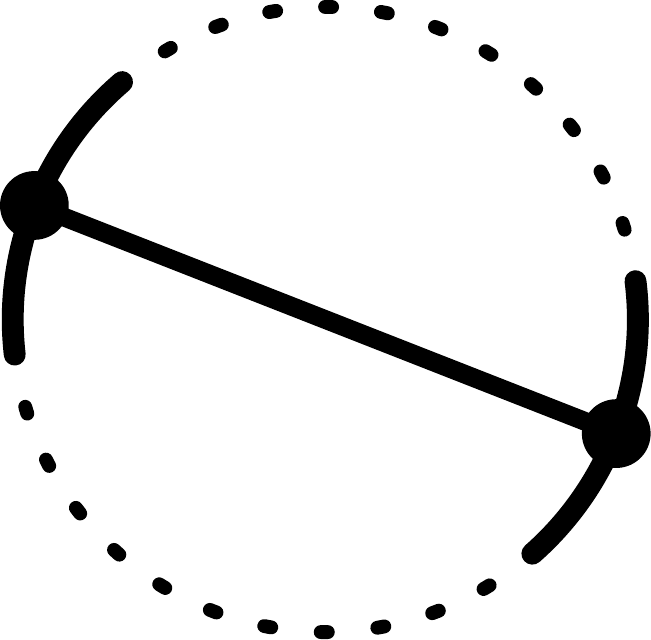}} 
        \put(13,15){\mbox{$\scriptstyle c$}}
      \end{picture}
    }    
    \rightsquigarrow \rb{\rrlen}{\ig[width=\rwlen]{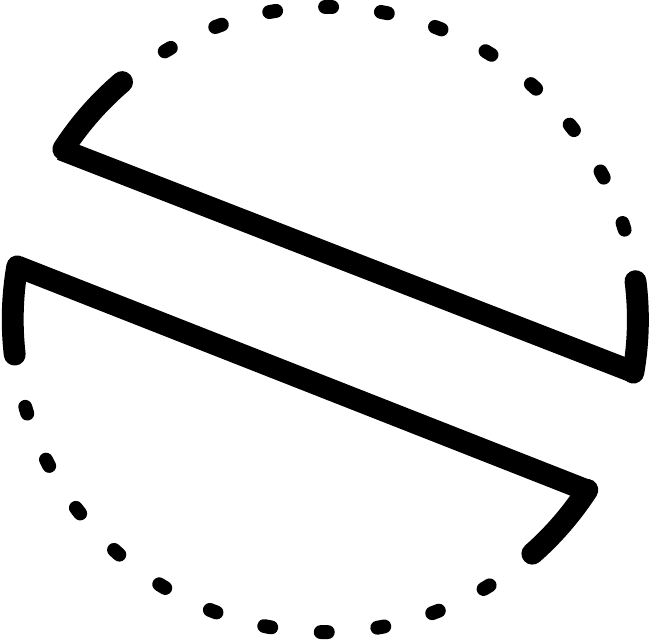}}
    \quad \text{if} \quad
    s(c) = 1 
    \qquad \text{and} \qquad 
    \rb{\rrlen}{
      \begin{picture}(27,27)(0,0)
        \put(0,0){\ig[width=\rwlen]{fig/1C}} 
        \put(13,15){\mbox{$\scriptstyle c$}}
      \end{picture}
    }    
    \rightsquigarrow \rb{\rrlen}{\ig[width=\rwlen]{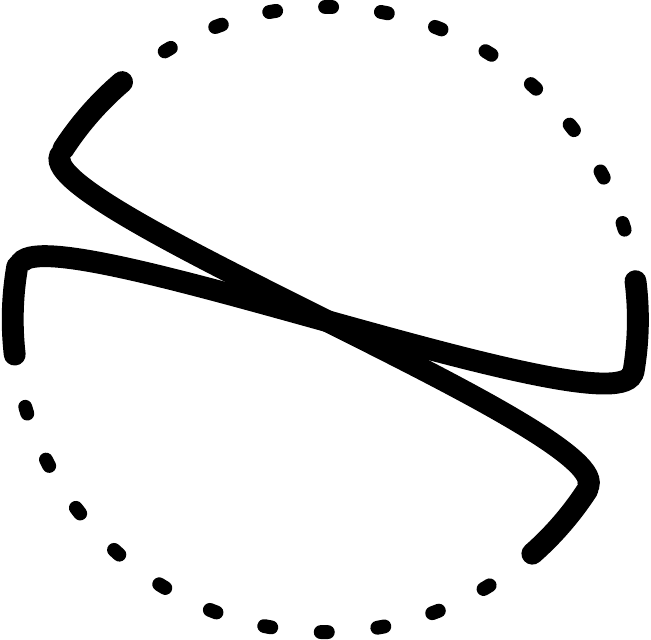}}
    \quad \text{if} \quad
    s(c) = -1.
  \end{align*}
  Then the symbol of $y_n$ is given by
  \begin{align*}
    \sigma^\ssf_n(y_n)(D) = \sum_s (-1)^{\abs{s}} 3^{c(D_s)},
  \end{align*}
  where the sum is over all maps $s$ on the chords of $D$, and $c(D_s)$
  is the number of components of the resolved diagram $D_s$.
\end{example}

The symbol of a framed invariant satisfies the 4T relation, but not
the 1T relation. Indeed, if $\fweights_n$ denotes the set of (framed)
weight systems, then the symbol descends to a map  
\begin{align}
  \label{eq:3}
  \overline{\sigma}^\ssf_n \colon \fvassinv_n / \fvassinv_{n-1} \to
  \fweights_n,
\end{align}
and \Fref{thm:1} has the following analogue.
\begin{theorem}
  \label{thm:3}
  There exists a map $J^\ssf_n \colon \fweights_n \to \fvassinv_n$ such
  that $\sigma^\ssf_n \circ J^\ssf_n = \Id$.
\end{theorem}

The first proof of this is due to Le and Murakami \cite{MR1394520}, and
employs a combinatorial description using the Drinfeld associator. An
alternative proof, using a framed version of the Kontsevich integral,
is provided by Goryunov \cite{MR1738389}. Once again, invariants in
the image of $J^\ssf$ are called canonical. 

Through the work of Reshetikhin and Turaev \cite{MR939474,MR1036112},
representations of quantum groups provide a rich source of framed knot
invariants. This elaborate construction associates an invariant
of framed knots to any representation of a semi-simple Lie algebra
\nolinebreak[4] $\frakg$. In fact, the representation of $\frakg$ can be deformed to a
representation of the quantum group $U_q(\frakg)$, yielding a solution
to the Yang-Baxter equation which is then used in constructing the
invariant.
\begin{example}
  \label{exa:3}
  For the standard representation of $\sl_2$, the resulting quantum
  invariant $\qinv^\ssf_{\sl_2}$ can be characterized by the skein
  relation
  \begin{align*}
    q^{\frac{1}{4}}
    \qinv^\ssf_{\sl_2}(\rb{\rrlen}{\ig[width=\rwlen]{fig/pcross}}) 
    - q^{\ssm \frac{1}{4}}
    \:\!\qinv^\ssf_{\sl_2}(\rb{\rrlen}{\ig[width=\rwlen]{fig/ncross}}) =
    (q^{\frac{1}{2}} - q^{\ssm \frac{1}{2}}) \:\!  \qinv^\ssf_{\sl_2}
    (\rb{\rrlen}{\ig[width=\rwlen]{fig/twovertical}}) 
  \end{align*}
  with the framing and initial conditions
  \begin{align*}
    \qinv^\ssf_{\sl_2} (\rb{\rrlen}{\ig[width=\rwlen]{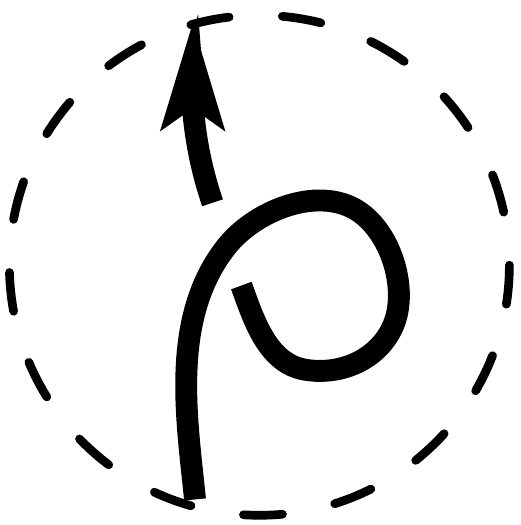}}) =
    q^{\frac{3}{2}} \, \qinv^\ssf_{\sl_2}
      (\rb{\rrlen}{\ig[width=\rwlen]{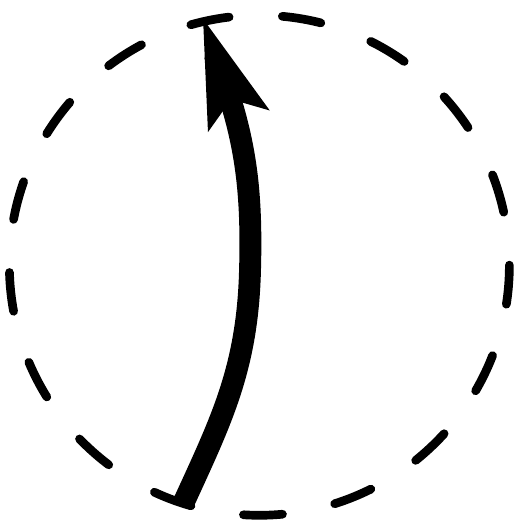}}) \quad \;\;
      \text{and} \quad \;\;\qinv^\ssf_{\sl_2}
      (\rb{\rrlen}{\ig[width=\rwlen]{fig/unknot}}) =
      {q^{\frac{1}{2}}+q^{\ssm \frac{1}{2}}}.
  \end{align*}
  There is a standard way of deframing general quantum invariants to
  produce invariants of unframed knots. For $\qinv^\ssf_{\sl_2}$ this
  deframing procedure amounts to
  \begin{align*}
    \qinv_{\sl_2}(K) = q^{-3/4\cdot w(K)} \qinv^\ssf_{\sl_2}(K),
  \end{align*}
  where $K$ is a knot diagram and $w(K)$ is its writhe, or simply the
  difference between the number of positive and negative crossings in
  the diagram. This means that the deframed invariant satisfies
  \begin{align*}
    q \;\!\qinv_{\sl_2} (\rb{\rrlen}{\ig[width=\rwlen]{fig/pcross}}) -
    q^{\ssm 1}
    \:\!\qinv_{\sl_2}(\rb{\rrlen}{\ig[width=\rwlen]{fig/ncross}}) =
    (q^{\frac{1}{2}} - q^{\ssm \frac{1}{2}}) \:\!  \qinv_{\sl_2}
    (\rb{\rrlen}{\ig[width=\rwlen]{fig/twovertical}}),
  \end{align*} 
  so in fact the Jones polynomial of a knot $K$ can be
  expressed as 
  \begin{align*}
    J(K) = (t^{\frac{1}{2}}+t^{\ssm \frac{1}{2}})^{\ssm 1}
    \!\;\qinv_{\sl_2} (K) \vert_{q = t^{- 1}}.
  \end{align*}
  In this sense, the quantum invariant $\qinv^\ssf_{\sl_2}$ constitutes a
  natural framed version of the Jones polynomial. 
\end{example}

\begin{example}
  \label{exa:7}
  The quantum invariant $\qinv^\ssf_{\so_3}$ arising from the
  standard representation of $\so_3$ is equal to the Yamada polynomial
  described in \Fref{exa:6} (see \cite{MR1854694}). 
\end{example}

It is a general result of Birman and Lin \cite{MR1198809} that any
quantum invariant produces a power series Vassiliev invariant by
substituting $q = e^h$ and expanding in $h$. We saw examples of this
in \Fref{exa:2} and \Fref{exa:6}. Moreover, the work of Le-Murakami
and Kassel \cite{MR1394520,MR1321145} establishes that these
finite-type invariants are in fact canonical. In particular, this
holds for the Yamada polynomial and for the Jones polynomial, as we
have already mentioned.

\subsection*{The Universal Vassiliev Invariant}

The universal Vassiliev invariant gives a unified way of encoding the
maps $J^\ssf_n$ of \Fref{thm:3}. We shall focus on the framed case,
but the description works equally well in the unframed setting. To
describe the universal Vassiliev invariant, it is convenient to
consider the 4T relation on the vector space spanned by chord
diagrams.

\begin{definition}
  \label{def:2}
  The space of (framed) chord diagrams $\fcda_n$ of order $n$ is the
  vector space $\setC \!\!\: A_n$, spanned by diagrams in $A_n$,
  modulo the subspace spanned by the 4T relations
  \begin{align}
    \label{eq:4Ta}
    \rb{\rrlen}{\ig[width=\rwlen]{fig/4T1}} \; - \;
    \rb{\rrlen}{\ig[width=\rwlen]{fig/4T2}} \; + \;
    \rb{\rrlen}{\ig[width=\rwlen]{fig/4T4}} \; - \;
    \rb{\rrlen}{\ig[width=\rwlen]{fig/4T3}} \; = \; 0.
  \end{align}
  The space of unframed chord diagrams is the vector space $\cda =
  \fcda / (\Theta)$, where $(\Theta)$ is the subspace spanned by the
  1T relations,
  \begin{align}
    \label{eq:1Ta}
    \rb{\rrlen}{\ig[width=\rwlen]{fig/1T}} = 0.
  \end{align}
\end{definition}
The space $\fcda = \bigoplus_n \fcda_n$ has a natural graded
multiplication given by connected sum of chord diagrams, which is only
well-defined up to the 4T relation. With this multiplication, the space
$(\Theta)$ above is exactly the ideal generated by the chord diagram
$\Theta$ with a single chord. In addition, the algebra $\fcda$ has
a graded coproduct given by
\begin{align}
  \label{eq:4}
  \Delta(D) = \sum_{J \subseteq \{D\}} D_J \otimes D_{\bar J},
\end{align}
where the sum is over all subsets $J$ of chords in $D$, and $D_J$ only
has chords of $J$, whereas $D_{\bar J}$ only has the complementary
set of chords. This gives $\fcda$ the structure of a Hopf algebra.

The dual of $\fcda_n$ is clearly just the space of weight systems
$\fweights_n = \Hom(\fcda_n, \setC)$, and through this duality 
the Hopf algebra structure on $\fcda$ induces a dual structure of the above
type on $\fweights = \bigoplus_n\! \fweights_n$. The
isomorphisms in \eqref{eq:3} respect the algebraic structure in the
sense that $\fweights$ is isomorphic to the associated graded of the filtered
bialgebra $\fvassinv$,
\begin{align*}
  \overline \sigma^\ssf \, \colon \bigoplus \fvassinv_n /
  \fvassinv_{n-1} \stackrel{\!\sim}{\longrightarrow} \bigoplus \fweights_n.
\end{align*}

Through the duality between $\fcda_n$ and $\fweights_n$, the map $J^\ssf_n
\colon \fweights_n \to \fvassinv_n$ of \Fref{thm:3} is equivalent to a
map
\begin{align*}
  I^\ssf_n \colon \fknots \to \fcda_n, \qquad \qquad J^\ssf_n(w)(K) =
  w(I^\ssf_n (K)). 
\end{align*}
Actually, the proof of \Fref{thm:3} constructs this map $I^\ssf_n$
rather than $J^\ssf_n$. Since $J^\ssf_n$ takes values in
$\fvassinv_n$, the map $I^\ssf_n$ is a Vassiliev invariant of order $n$,
with values in $\fcda_n$. The maps $I^\ssf_n$ are
usually collected in a single map, called the \emph{universal Vassiliev
  invariant},
\begin{align*}
  I^\ssf \colon \fknots \to \fcdac,
\end{align*}
where $\fcdac = \prod_n \!\fcda_n$ is the graded completion of the
algebra $\fcda$ of chord diagrams. Appropriately normalized, this map
is actually multiplicative with respect to the connected sum of knots
(see \cite{MR2962302}). Moreover, the coefficients of $I^\ssf$ are in
fact rational \cite{MR1394520}.

By a standard symbol calculus
argument, it is easy to show that any Vassiliev invariant factors
through the universal one.
%\pagebreak[2]
\begin{proposition}
  \label{prop:1}
  For any $v \in \fvassinv$, there is a unique $w \in \fweights$ such
  that $v = w \circ I^\ssf$.
\end{proposition}
In this sense, the universal Vassiliev invariant allows us to speak
about the lower order symbols of any Vassiliev invariant in
$\fvassinv$. Moreover, applying the maps $J^\ssf_n$ to this total
symbol, any Vassiliev invariant can be uniquely written as
\begin{align*}
  v = v^c_n + \cdots + v^c_1 + v^c_0,
\end{align*}
where $v^c_i \in \fvassinv_i$ are canonical Vassiliev invariants. In
fact, this establishes a bijective correspondence between the subspace of
$\fvassinvc$ consisting of canonical power series Vassiliev invariants
and weight systems in the graded completion $\fweightsc = \prod_n
\!\fweights_n$.

\chapter{Constructing Weight Systems}
\label{cha:constr-weight-syst}

In this section, we shall use the geometry of pseudo-Riemannian symmetric
spaces to construct framed weight systems on chord diagrams. Through
the Kontsevich integral, these will in turn give rise to Vassiliev
invariants. As we shall see, the construction does not yield new weight
systems, but rather weights of the familiar Lie algebra type.

We start by describing a simple general recipe for constructing weight
systems.  Suppose that $V$ is a finite-dimensional vector space and
consider a symmetric tensor,
\begin{align*}
  H \in \End(V) \otimes \End(V).
\end{align*}
Symmetry means that $H$ is invariant under the involution that
interchanges the two copies of the endomorphism space, or simply
$H^{bd}_{ac} = H^{db}_{ca}$ in index notation. Such a tensor can be
used to construct a function $w_H \colon A \to \setC$ on chord
diagrams. The value on a particular chord diagram $D \in A_n$ is given
by a full contraction of $n$ copies of $H$ as prescribed by the chords
of $D$. This can be explained using diagram by representing
the tensor $H$ in the following graphical way
\begin{align}
  \label{eq:5}
  H^{bd}_{ac} \; = \; \rb{-7.7mm}{\ig[height=19mm]{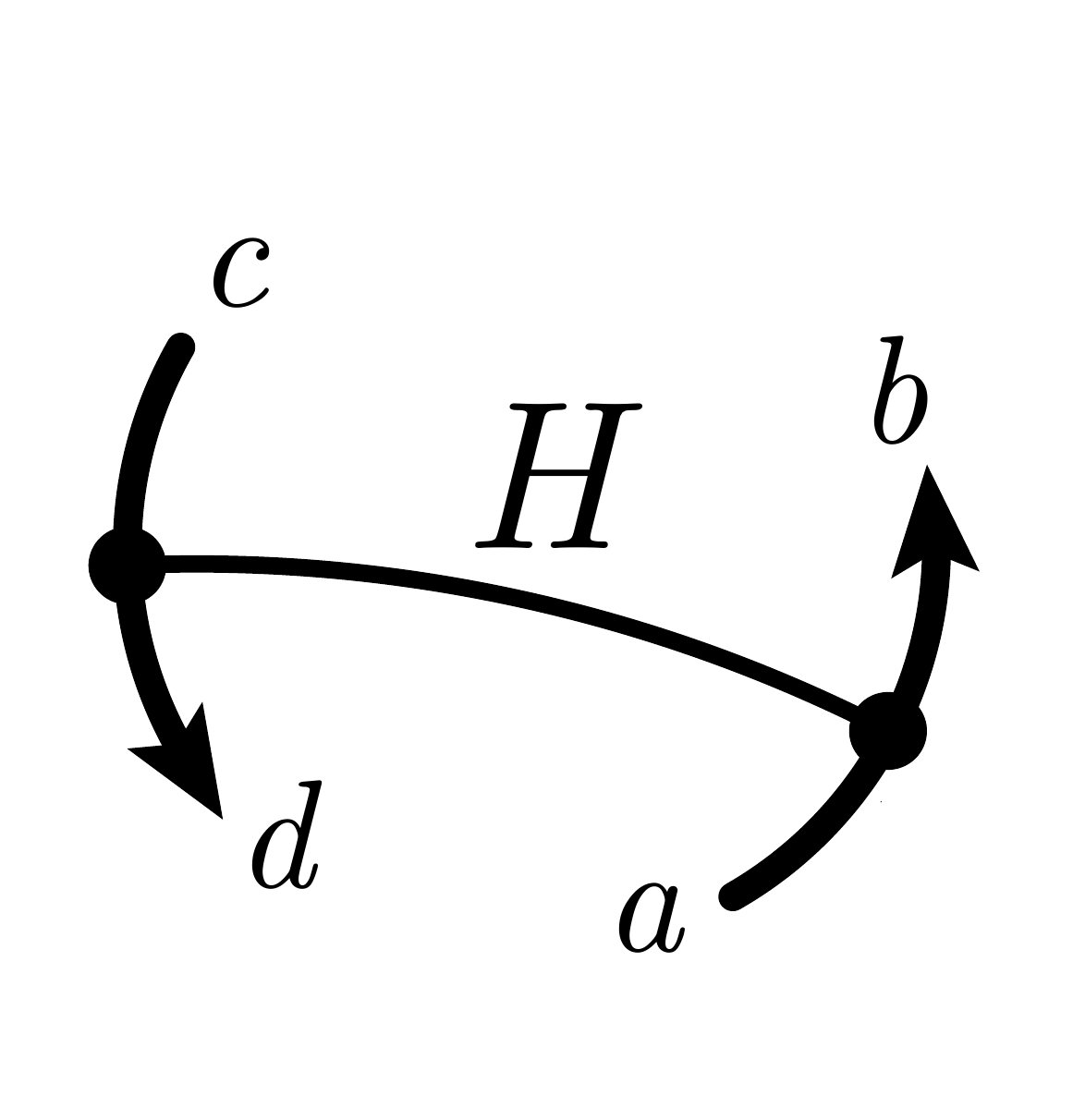}}
  \qquad \qquad 
  \text{or simply} \qquad \qquad H \; = \;
  \rb{-7.7mm}{\ig[height=19mm]{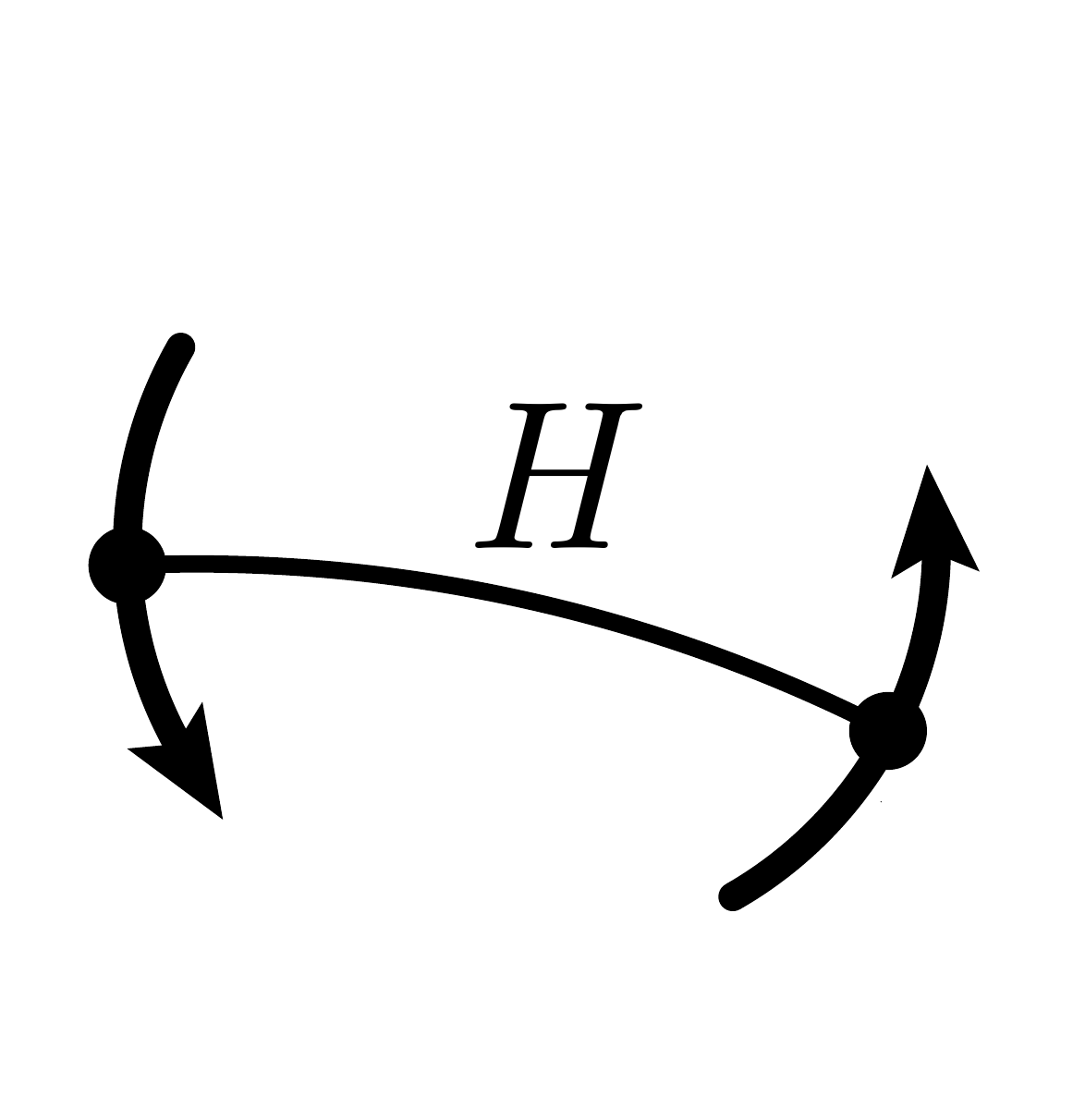}}. 
\end{align}
Given a chord diagram, we view it as being glued from such pieces, and
contract the tensors accordingly. For example, we have
\begin{align*}
  w_H (\rb{\rrlen}{\ig[width=\rwlen*105/100]{fig/HChord}}) \quad  = \quad
  \rb{-9.9mm}{
    \begin{picture}(52,62)(0,0)
      \put(0,0){\ig[height=22mm]{fig/RContract}} 
      \put(23,33){\mbox{$\scriptstyle H$}}
      \put(38,37){\mbox{$\scriptstyle H$}}
      \put(8,21){\mbox{$\scriptstyle H$}}
    \end{picture}
  } 
  \quad = \quad
  H^{bf}_{ae}H^{ad}_{cf}H^{ce}_{bd}, 
\end{align*}
where repeated indices denote contraction of tensor entries, or simply
Einstein summation if we think of the indices as indexing a basis of
the vector space $V$. The symmetry of the tensor $H$ ensures that the
diagrammatic representation is well-defined.

Using the following graphical representation, the map $w_H$ is even
more naturally expressed on singular knots, where the tensor $H$ can
be placed at every double point and contracted as prescribed by the knot,
\begin{align}
  \label{eq:7}
  H \;\; = \;\; \rb{-4.25mm}{\ig[height=11mm]{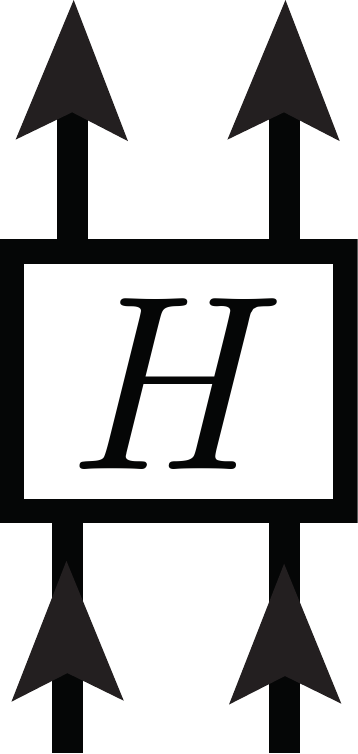}} \qquad \qquad
  \text{and} \qquad \qquad w_H \left
    (\rb{-7.5mm}{\ig[height=17mm]{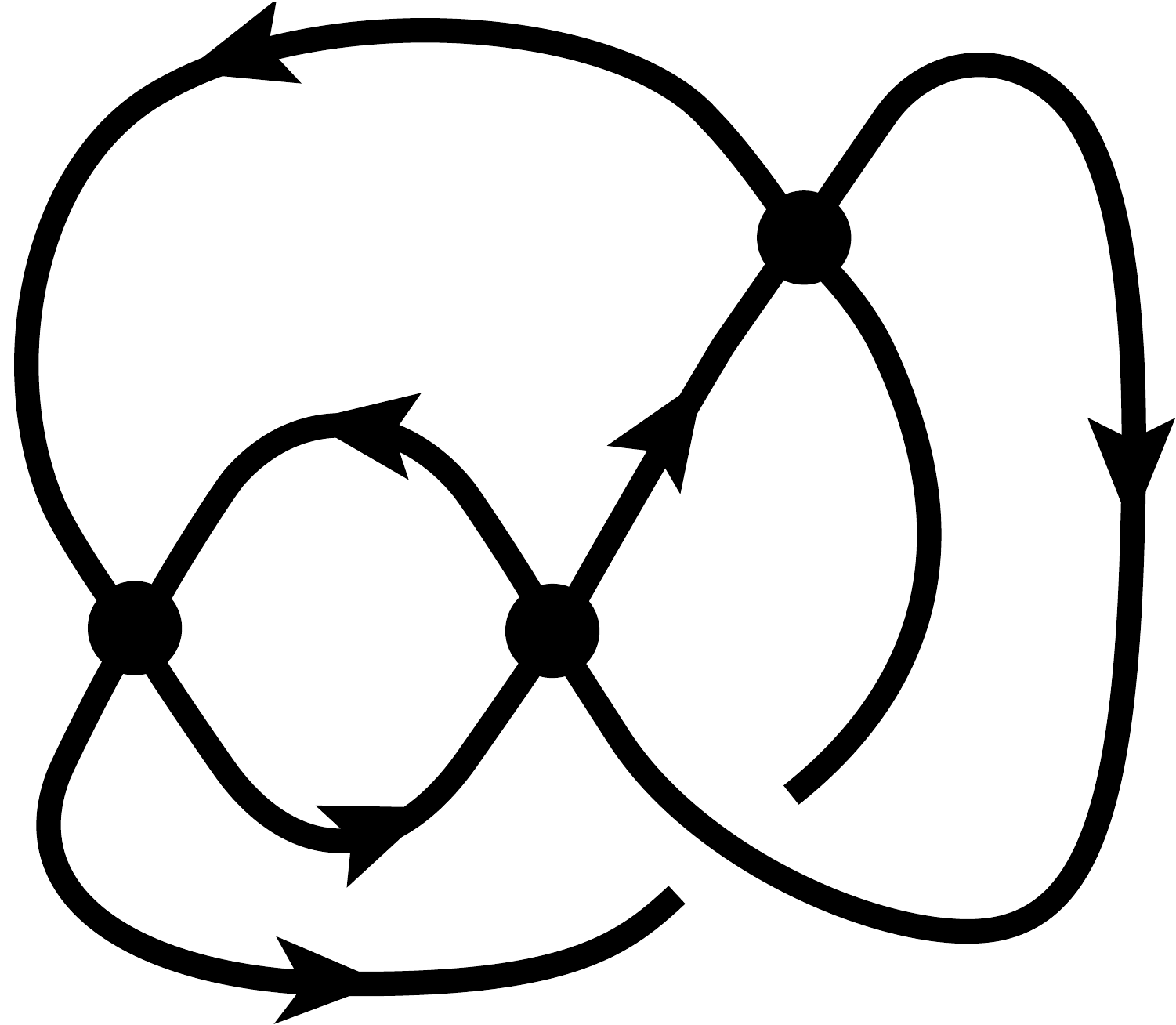}} \right) \quad = \quad
  \rb{-7.5mm}{\ig[height=17mm]{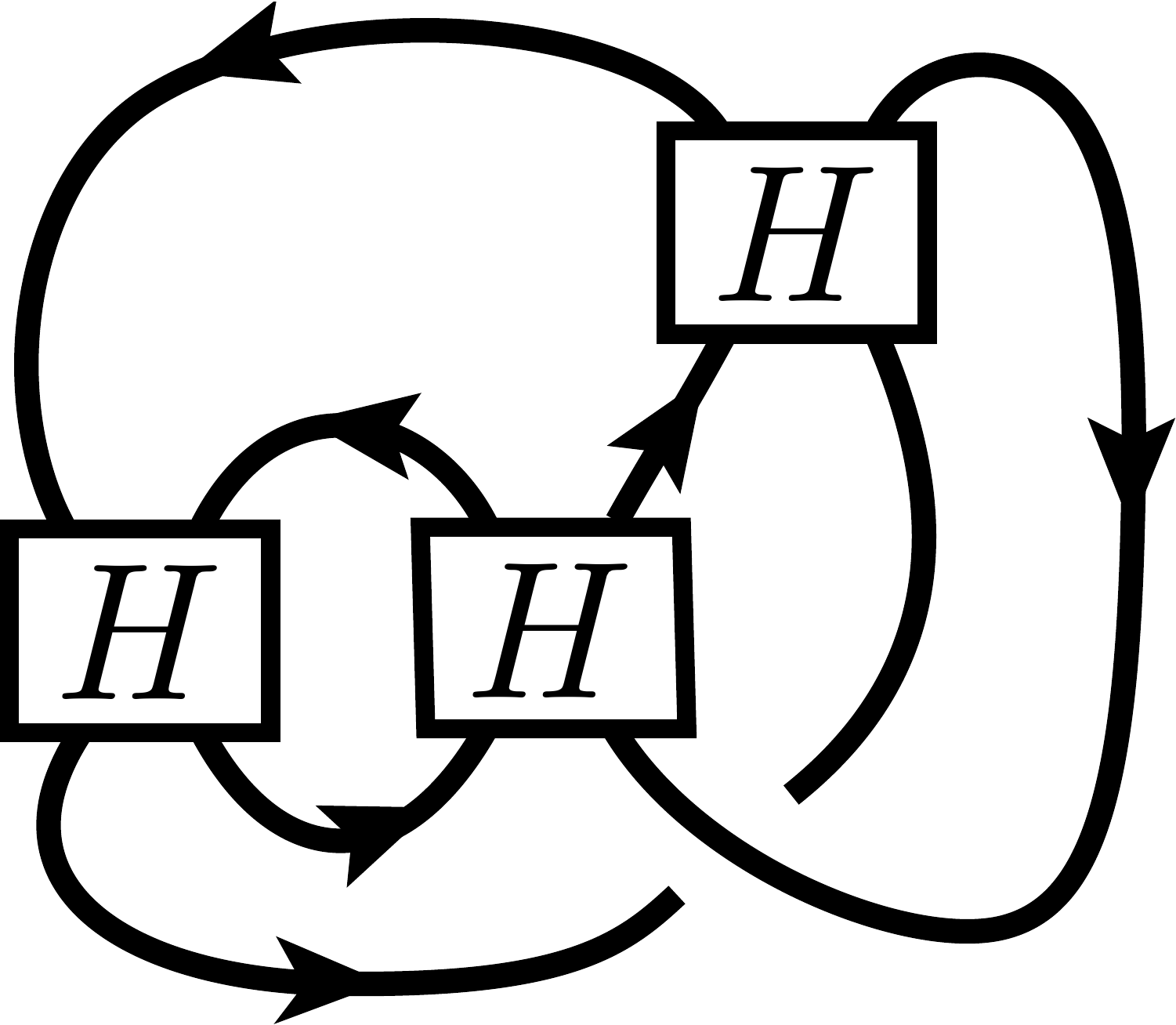}}.
\end{align}
For a general tensor $H$, the function $w_H$ will not necessarily
satisfy the 4T relation and define a weight system.

\begin{definition}
  \label{def:3}
  A \emph{weight tensor} is a symmetric element $H \in \End(V)
  \otimes \End(V)$ which satisfies the 4T relation
  \begin{align}
    \label{eq:6}
    H^{fx}_{ea}H^{bd}_{xc} - H^{fb}_{ex} H^{xd}_{ac} +
    H^{fx}_{ec}H^{bd}_{ax} - H^{fd}_{ex}H^{bx}_{ac}
    = 0     
  \end{align}
  as an element in $\End(V) \otimes \End(V) \otimes \End(V)$.
\end{definition}
\noindent In terms of the graphical representation in \eqref{eq:5},
the 4T relation corresponds to
\begin{align}
  \label{eq:4TH}
  \rb{\rrlen*110/100}{\ig[width=\rwlen*110/100]{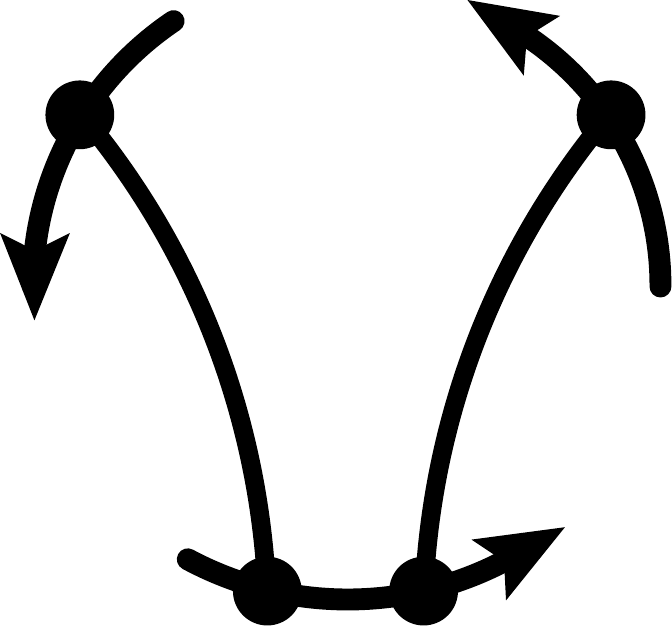}} \; - \;
  \rb{\rrlen*110/100}{\ig[width=\rwlen*110/100]{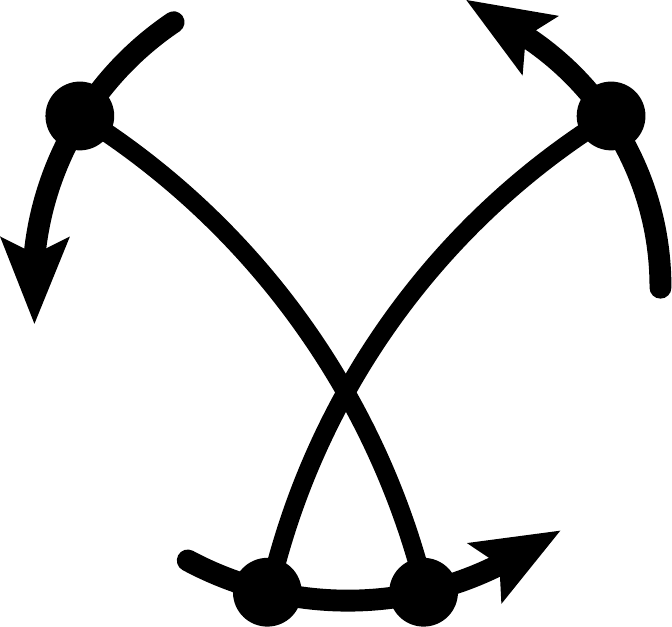}} \; + \;
  \rb{\rrlen*110/100}{\ig[width=\rwlen*110/100]{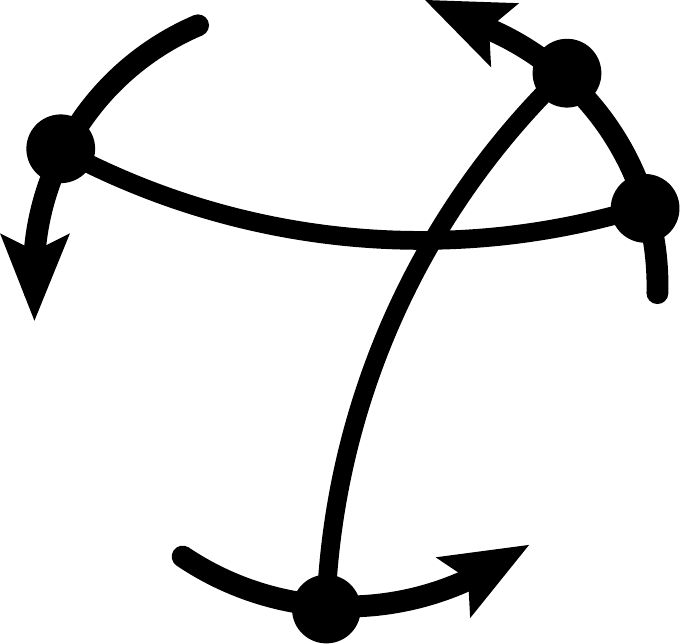}} \; - \;
  \rb{\rrlen*110/100}{\ig[width=\rwlen*110/100]{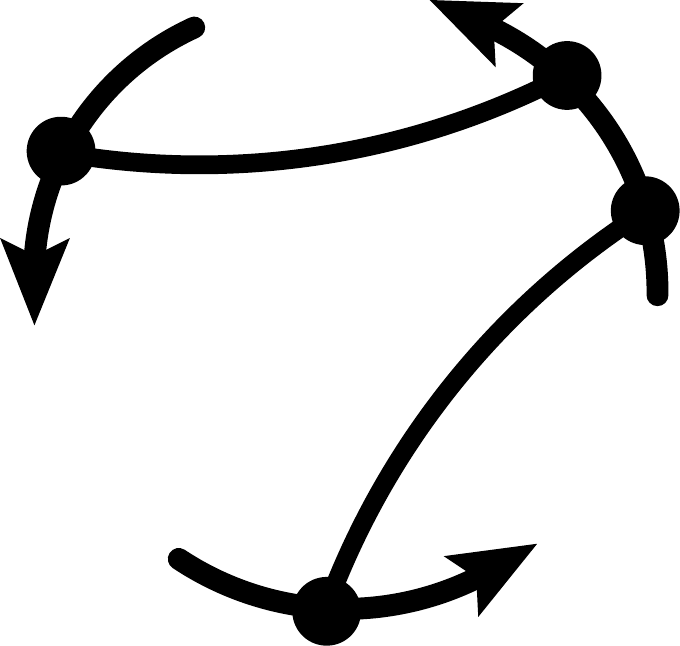}} \; = \; 0.
\end{align}
The point of this definition is, of course, the following obvious statement:

\begin{proposition}
  \label{prop:3}
  Any weight tensor $H$ defines a framed weight system $w_H \colon
  \fcda \to \setC$.
\end{proposition}

\subsection*{Weight Systems from Metrized Lie Algebras}

There is a standard way of producing a framed weight system from a
representation of a metrized Lie algebra. The construction is due to
Bar-Natan \cite{MR1318886} and proceeds as follows.

Let $\frakg$ be a Lie algebra equipped with a non-degenerate bilinear
form $B\in \frakg^* \otimes \frakg^*$ which is invariant under the
adjoint action, i.e.,
\begin{align*}
  B( [z, x], y ) + B( x, [z, y] ) = 0, 
\end{align*}
for all $x,y,z \in \frakg$. Since $B$ is non-degenerate, its inverse
constitutes an element $C \in \frakg \otimes \frakg$, called the
Casimir tensor. If $$\rho \colon \frakg \to \End(V)$$ is a
representation of $\frakg$, then its tensor square can be applied to
$C$ to produce a tensor
\begin{align*}
  \rho(C) \in \End(V) \otimes \End(V).  
\end{align*}
This is a weight tensor. Indeed, if $Y \in \frakg^{\otimes 3}$ denotes
the totally anti-symmetric structure tensor of $\frakg$, given by
raising both indices on the Lie bracket using the metric, then
$\rho(C)$ satisfies the 4T relation because
\begin{align}
  \label{eq:15}
  \rho(C)^{xf}_{ae}\rho(C)^{bd}_{xc} -
  \rho(C)^{xd}_{ac}\rho(C)^{bf}_{xe} = \rho(Y)^{bdf}_{ace} =
  \rho(C)^{bx}_{ac}\rho(C)^{df}_{xe} - \rho(C)^{bd}_{ax}\rho(C)^{xf}_{ce}.
\end{align}
Graphically, this corresponds to
\begin{align}
  \label{eq:4TL}
  \rb{\rrlen*110/100}{\ig[width=\rwlen*110/100]{fig/4T1-H}} \; - \;
  \rb{\rrlen*110/100}{\ig[width=\rwlen*110/100]{fig/4T2-H}} \;\; = \;\;
  \rb{\rrlen*110/100}{\ig[width=\rwlen*110/100]{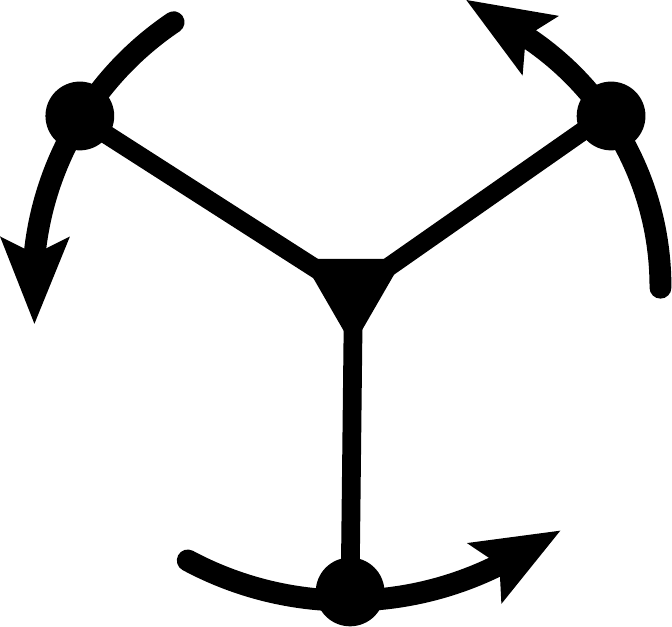}} \;\; = \;\;  
  \rb{\rrlen*110/100}{\ig[width=\rwlen*110/100]{fig/4T3-H}} \; - \;
  \rb{\rrlen*110/100}{\ig[width=\rwlen*110/100]{fig/4T4-H}}
\end{align}
where the trivalent vertex in the middle represents
the trivector $Y \in \frakg^{\otimes 3}$. From \Fref{prop:3}, we
immediately get the following:

\begin{proposition}
  \label{prop:4}
  Any representation $\rho \colon \frakg \to \End(V)$ of a metrized
  Lie algebra defines a framed weight system $w_{\rho} \colon \fcda
  \colon \to \setC$.
\end{proposition}

The weight systems obtained in this way are said to be of Lie algebra
type. A more complicated way of obtaining a weight system from a Lie
algebra representation is through the corresponding quantum
invariant. By the result of Birman and Lin \cite{MR1198809}, the
quantum invariant defined from a representation of a Lie algebra
$\frakg$ gives rise to a power series Vassiliev invariant through the
substitution $q = e^h$. Conveniently, the weight system of this power
series invariant matches the weight system constructed directly from
the Lie algebra representation \cite{MR1321294}.

\begin{example}
  \label{exa:4}
  The Lie algebra $\sl_2$ has the standard generators
  \begin{align*}
    H = \begin{pmatrix} 1 & 0 \\ 0 & -1 \end{pmatrix}
    \qquad \quad
    E = \begin{pmatrix} 0 & 1 \\ 0 & 0 \end{pmatrix}
    \qquad \quad 
    F = \begin{pmatrix} 0 & 0 \\ 1 & 0 \end{pmatrix},
  \end{align*}
  and the Casimir tensor of the non-degenerate form $B(x,y) = \Tr(xy)$
  is given by
  \begin{align}
    \label{eq:14}
    C = \frac{1}{2} H \otimes H + E \otimes F + F \otimes E.
  \end{align}
  The weight tensor associated with the standard representation is of
  course given by the same expression, and it agrees with the symbol
  of the power series Vassiliev invariant coming from the quantum
  invariant $\qinv^\ssf_{\sl_2}$, as discussed in \Fref{exa:3}. A
  combinatorial description of this weight system, along the lines of
  \Fref{exa:8}, is given in \cite{MR2962302}.
\end{example}

\begin{example}
  \label{exa:9}
  The weight system arising from the standard representation of the
  Lie algebra $\so_3$, with non-degenerate and invariant bilinear form
  $B(x,y) = \frac{1}{2} \Tr(xy)$, is equal to the weight system of the
  Yamada polynomial, described in \Fref{exa:6} and \Fref{exa:8} (see
  \cite{MR1318886,MR2962302}). This is in agreement with Birman and
  Lin \cite{MR1198809} and the fact that the Yamada polynomial
  corresponds to the quantum invariant coming from $\so_3$ with its
  standard representation \cite{MR1854694}.
\end{example}

\subsection{Weight Systems from Symmetric Spaces}

We shall see that pseudo-Riemannian symmetric spaces give rise to a framed
weight system in the simplest way possible: the curvature defines a
weight tensor. 

Let $(M, g)$ denote a pseudo-Riemannian manifold, with
Levi-Civita connection $\nabla$ and curvature $R$, given as usual by
\begin{align*}
  R(X,Y)Z = \nabla_X \nabla_Y Z - \nabla_Y \nabla_X Z - \nabla_{[X,Y]} Z,
\end{align*}
for any vector fields $X,Y,Z$ on $M$.

\begin{definition}
  \label{def:4}
  A (locally) \emph{symmetric space} is a connected pseudo-Riemannian
  manifold $(M, g)$ with parallel curvature tensor, that is $\nabla R = 0$.
\end{definition}

The justification of this name lies in the standard fact that the curvature
is parallel if and only if the geodesic reflections are local
isometries. A special case is when these reflections extend to global
isometries of the manifold.

\begin{definition}
  \label{def:5}
  A \emph{globally symmetric space} is a pseudo-Riemannian manifold $(M,
  g)$ such that every point $p \in M$ is an isolated fixed point of an
  involutive isometry $s_p$.
\end{definition}

In addition to the obvious
antisymmetry $R^d_{abc} = - R^d_{bac}$, the curvature tensor satisfies the algebraic
Bianchi identity
\begin{align*}
  R^d_{abc} +  R^d_{bca} + R^d_{cab} = 0,
\end{align*}
as well as the differential Bianchi identity on the covariant
derivative of $R$, which is vacuous in the case of symmetric spaces
since the curvature is parallel.

To expose the remaining symmetries of the curvature, it is common to
lower the index and consider the fully covariant curvature tensor
\begin{align*}
  \check R_{abcd} = R^x_{abc} g_{xd} 
  \qquad \qquad \text{which satisfies}  \qquad \qquad 
  \check R_{abcd} = \check R_{cdab}. %= - \check R_{bacd}
\end{align*}
We shall, however, be interested in the tensor $\hat R \in C^\infty(M,
\End(TM) \otimes \End(TM))$ given by raising an index of $R$,
\begin{align*}
  \hat R^{bd}_{ac} = g^{bx} R^d_{axc}.
\end{align*}
The symmetry of $\check R$ translates to the symmetry $\hat
R^{bd}_{ac} = \hat R^{db}_{ca}$, so $\hat R$ defines a symmetric
section of $\End(TM) \otimes \End(TM)$. Moreover, we have the
following simple but crucial result, which is the main
observation of the paper.

\begin{proposition}
  \label{prop:2}
  The curvature $\hat R$ of a locally symmetric space
  is a weight tensor.
\end{proposition}

\begin{proof}
  We must verify the 4T relation \eqref{eq:6}. Since the curvature
  tensor is parallel, we have
  \begin{align}
    0 
    &= \nabla_f \nabla_e R^d_{abc} - \nabla_e \nabla_f R^d_{abc}
    \nonumber \\ 
    \label{eq:10}
    &= R^x_{efa} R^d_{xbc} + R^x_{efb} R^d_{axc} 
    + R^x_{efc} R^d_{abx} - R^d_{efx} R^x_{abc}.
  \end{align}
  Raising the indices $b$ and $f$ with the metric, we obtain 
  \begin{align*}
    0 &= \hat R^{fx}_{ea} \hat
    R^{bd}_{xc} - \hat R^{fb}_{ex} \hat R^{xd}_{ac} + \hat R^{fx}_{ec}
    \hat R^{bd}_{ax} - \hat R^{fd}_{ex} \hat R^{bx}_{ac},
  \end{align*}
  which is the desired identity.
\end{proof}

The main theorem follows as an immediate corollary of the above
proposition.

\begin{theorem}
  \label{thm:5}
  The curvature $\hat R$ of a locally symmetric space defines a framed
  weight system $w_R \colon \fcda \to \setR$.
\end{theorem}

\begin{proof}
  By \Fref{prop:2}, the curvature tensor $\hat R$ satisfies the
  4T relation on the tangent space at every point on the symmetric
  space $M$. This means that $\hat R$ defines
  a map $w_R \colon \fcda \to C^\infty(M)$. For a given chord diagram
  $D \in \fcda_n$, the function $w_R(D)$ is given by a full
  contraction of the curvature and the metric, both of which are
  parallel. This means that the derivative of $w_R(D)$ must vanishes
  so that $w_R$ takes values in constant functions.
\end{proof}

Through these weight systems and the universal Vassiliev invariant,
any symmetric space gives rise to a power series Vassiliev invariant.

\begin{remark}
  \label{rem:1}
  In fact, the proof of \Fref{prop:2} does not require the curvature
  to be parallel, but only that its covariant derivative $\nabla R \in
  \Omega^1(M, (TM^{*})^{\otimes 3} \otimes TM)$ defines a closed
  one-form with respect to the covariant exterior derivative
  $d^{\nabla}$. It would be interesting to find examples of
  non-symmetric spaces of this type. The weight system coming from
  such a space would in general take values in functions on $M$, which
  could be integrated to produce numbers for compact spaces.
\end{remark}

In general, the weight systems coming from symmetric spaces will
detect the framing. Indeed, the value on the chord diagram with a
single chord will be given by the scalar curvature, so the weight
system will be framed unless the symmetric space is scalar flat. Since
any irreducible symmetric space is Einstein (\cite{MR2243772}, Theorem
57), it is only scalar flat if it is Ricci flat, in which case it is
flat altogether (\cite{MR867684}, Theorem 7.61).

\begin{example}
  \label{exa:5}
  The three-sphere $S^3 = \{x \in \setR^4 \mid \abs{x} = 1\}$, with
  the subspace metric $g$ induced from the standard Euclidean metric
  on $\setR^4$, is a simply connected Riemannian symmetric space. The
  sectional curvature is constant and equal to 1, so the curvature
  tensor is given by
  \begin{align}
    \label{eq:8}
    \check R_{abcd} = g_{ad}g_{bc} - g_{ac}g_{bd}
    \qquad \text{or equivalently} \qquad 
    \hat R^{bd}_{ac} = \idone^d_{a}\idone^b_{c} - g_{ac}g^{bd}.
  \end{align}
  This is easily seen to imply that the corresponding weight system is
  given by
  \begin{align*}
    w_R(D) = \sum(-1)^{\abs{s}}\, 3^{c(D_s)},
  \end{align*}
  which agrees with that of the Yamada polynomial described in
  \Fref{exa:8}. Indeed, the particular form of the curvature
  \eqref{eq:8} means that evaluating the corresponding weight system
  on a singular knot, as in \eqref{eq:7}, yields the signed sum over
  all ways of smoothing the double point horizontally and vertically
  and taking the trace of the identity endomorphism, which is equal to
  $3$, for each the resulting components.
\end{example}

It is no coincidence that the weight system for $S^3$ is of Lie
algebra type. As we shall see, this is in fact the case for any weight
system constructed from the curvature tensor of a symmetric space.

\chapter{Symmetric Spaces and Lie Algebras}
\label{cha:symmetric-spaces-lie}

Representations of Lie algebras enter the picture in a natural way for
globally symmetric spaces, which are homogeneous spaces for the action
of the isometry group. An isotropy subgroup has an infinitesimal
action of its Lie algebra on the tangent space at the fixed
point. Now, there are many ways of representing a homogeneous space as
a quotient $M = G/H$, and the connection between symmetric spaces and
Lie algebra weight systems is the fact that for such a representation
the curvature of the symmetric space can be calculated from the
bracket on the Lie algebra $\frakg$ of $G$ as 
\begin{align*}
  R(X,Y)Z = -[[X,Y],Z] \qquad \qquad \forall X,Y,Z \in T_p M \subset
  \frakh \oplus T_pM = \frakg\, .
\end{align*}

In general, however, the Lie algebra of $G$ is not metrized in any
natural way, and we need a metrized Lie algebra to produce weight
systems. It turns out that a more careful representation of the
symmetric space as homogeneous for the transvection group resolves
this issue. For positive-definite Riemannian symmetric spaces, the
transvection group coincides with the connected component of the
isometry group, but not in general. In this representation, the
stabilizer of a point is isomorphic to the holonomy group and its Lie
algebra carries a natural metric. The Lie algebra of the holonomy
group also has the advantage that it makes sense for locally symmetric
spaces and does not rely on the homogeneous description.

As we have seen, the weight system constructed from a symmetric space
is determined by the metric and the curvature on a single tangent
space, and does not depend on the global geometry of the symmetric
space. The Lie algebras of the transvection and holonomy groups can be
described directly from this linear data, avoiding the use of Lie
group theory for symmetric spaces. For the general theory of symmetric
spaces, the reader is referred to \cite{MR514561,MR2243772,MR2243012}
and in particular to \cite{MR556610,math/0301326,MR2577201} for the
specific theory used below.

Let $(M, g)$ be any locally symmetric space with
curvature $R$. Fix a base point $p \in M$, and denote the tangent
space at $p$ by $\frakp$. Let $\frakh$ be the subspace of
$\End(\frakp)$ defined by
\begin{align}
  \label{eq:9}
  \frakh = \spn \{ R(X,Y) \in \End(\frakp) \mid X,Y \in \frakp \}.
\end{align}
By the Ambrose-Singer theorem, $\frakh$ is in fact the Lie algebra of
the holonomy group at \nolinebreak[2] $p$. The Lie bracket is given by the usual
commutator of endomorphisms, which can be seen to preserve $\frakh$ by
rewriting \eqref{eq:10} as
\begin{align}
  \label{eq:11}
  [R(X,Y), R(Z,W)] = R(R(X,Y)Z,W) + R(Z,R(X,Y)W) 
\end{align}
for $X,Y,Z,W \in \frakp$. This can be viewed as a variant of the
4T relation.

Now, consider the sum $\frak g = \frakh \oplus \frakp$ and extend the
Lie bracket on $\frakh$ by
\begin{alignat*}{2}
  [X,Y] &= R(X,Y) & \qquad &\text{for} \quad X,Y \in \frakp \\
  [A,X] &= -[X,A] = A(X) & \qquad &\text{for} \quad A \in \frakh, X
  \in \frakp.
\end{alignat*}
This is clearly skew-symmetric. Moreover, it satisfies the Jacobi
identity, which reduces to \eqref{eq:11} for the case of $X,Y \in
\frakp$ and $A \in \frakh$, and to the Bianchi identity of $R$ for the
case of $X,Y,Z \in \frakp$. In other words, the bracket defines a Lie
algebra structure on  $\frakg$, and we have
\begin{align*}
  [\frakh, \frakh] \subset \frakh \qquad \qquad[\frakh, \frakp] \subset
  \frakp \qquad \qquad [\frakp, \frakp] = \frakh.
\end{align*}
Notice in particular that $\frakh$ equals $[\frakp, \frakp]$ by
definition, and that we get a representation of $\frakh$ on $\frakp$
by the adjoint action.

Using the splitting $\frakg = \frakh \oplus \frakp$, we define an
involution
\begin{align*}
  s \colon \frakg \to \frakg \qquad \qquad \text{by} \qquad \qquad s(A
  + X) = A - X \quad \text{for} \quad A \in \frakh, X \in \frakp. 
\end{align*}
The metric on $M$ defines a non-degenerate bilinear form $B_{\frakp}$
on $\frakp$, and this pairing is invariant under the adjoint action of
$\frakh$, by the symmetries of the curvature. Moreover, there is a
unique extension of $B_{\frakp}$ to $\frakg$ which is invariant under
the involution $s$ and the adjoint action of $\frakg$ on itself
\cite{MR556610}. By invariance under $s$, the summands $\frakh$ and
$\frakp$ must be orthogonal, and on $\frakh = [\frakp, \frakp]$ the
invariance under the adjoint action forces
\begin{align}
  \label{eq:13}
  B_{\frakh}([X,Y],[Z,W]) = B_{\frakp} ([[X,Y],Z],W) = \check R(X,Y,Z,W).
\end{align}
The extension $B = B_{\frakh} + B_{\frakp}$ is actually
non-degenerate. Indeed, if $A \in \frakh \subset \End(\frakp)$ and
\begin{align*}
  0 = B_{\frakh}(A, [X,Y]) = B_{\frakp}([A,X], Y) = B_{\frakp}(A(X),Y)
  \qquad \forall X,Y \in \frakp, 
\end{align*}
then $A$ must vanish by the non-degeneracy of $B_{\frakp}$.

In conclusion, we have a metrized Lie algebra $\frakh$ and a
representation $\rho \colon \frakh \to \End(\frakp)$ given by the
adjoint action on $\frakg$. The inverse of the non-degenerate metric
$B_{\frakh}$ defines an element $C_{\frakh} \in \frakh \otimes
\frakh$, and \eqref{eq:13} has the reformulation
\begin{align*}
  \rho(C_{\frakh}) = \hat R,
\end{align*}
as tensors in $\End(\frakp) \otimes \End(\frakp)$. This proves
the following:

\begin{theorem}
  \label{thm:6}
  Weight systems from symmetric spaces are of Lie algebra type.
\end{theorem}

Theorem \ref{thm:6} explains the observation in \Fref{exa:5} above:

\begin{example}
  \label{exa:10}
  The round unit sphere $S^3$ has isometry group $O(4)$ with isotropy
  group $O(3)$. The transvection group is the connected component containing the
  identity, namely  $SO(4)$, and its isotropy subgroup $SO(3)$ gives the
  holonomy. Therefore, the holonomy Lie algebra is $\so_3$, which
  explains the equality of its weight system with that of the Yamada
  polynomial.
\end{example}

\subsection*{Realizing Weight Systems on Symmetric Spaces}

To understand which Lie algebra weight systems can be realized on
symmetric spaces, we shall introduce the following terminology
\cite{MR556610}.

\begin{definition}
  \label{def:6}
  A \emph{symmetric triple} $(\frakg, s, B)$ consists of 
  \begin{enumerate}\firmlist
  \item[-] a finite-dimensional real Lie algebra $\frakg$,
  \item[-] an involutive automorphism $s \colon \frakg \to \frakg$, and
  \item[-] a non-degenerate, symmetric bilinear form $B$, which is
    invariant in the sense that
    \begin{align*}
      \qquad B([x,y], z) = B(x,[y,z])  \qquad \text{and} \qquad B(sx, sy)
      = B(x,y) \qquad \quad \forall x,y,z \in \frakg.
    \end{align*}
  \end{enumerate}
  %\nopagebreak[4]
  Furthermore, the decomposition $\frakg = \frakh \oplus \frakp$ into
  eigenspaces of $s$ must satisfy $\frakh = [\frakp, \frakp]$.
\end{definition}
\pagebreak[2]

As described above, any symmetric space defines a symmetric triple. On
the other hand, given a symmetric triple $(\frakg, s, B)$, one can
construct a globally symmetric space in the following way. Let $G$ be
a simply connected Lie group with Lie algebra $\frakg$, and let $S \in
\Aut(G)$ be the automorphism with derivative $s$. The space $H =
G^{S}\subset G$ of fixed points of $S$ is closed and connected, and therefore
the quotient manifold $M = G/H$ is simply connected by the long exact
homotopy sequence. The form $B_\frakp$ is invariant under the
adjoint action of $H$ on $\frakp$, and so it defines a $G$-invariant
metric on $M$. This gives $M$ the structure of a symmetric space, and
the associated symmetric triple is isomorphic to the original
$(\frakg, s, B)$. In fact, the constructions realize the following
correspondence \nolinebreak[2]\cite{MR556610}.

\begin{theorem}
  \label{thm:7}
  There is a bijective correspondence between symmetric triples and
  simply connected globally symmetric spaces.
\end{theorem}

Given a weight system defined by a representation $$\rho \colon \frakg
\to \End(V)$$ of a metrized real Lie algebra $\frakg$, the
constructions above tell us how the weight tensor $\rho(C) \in \End(V)
\otimes \End(V)$ might be realized on a symmetric space. The point is
that the weight tensor knows about the holonomy Lie algebra. There must
exist some non-degenerate bilinear form $B_V$ on $V$ such that by
lowering the indices of $\rho(C)$, we obtain a tensor ${\check R^\rho}
\in (V^*)^{\otimes 4}$ with all the symmetries of a curvature
tensor. First of all, it must satisfy
\begin{align*}
  \check R^\rho_{cdab} = \check R^\rho_{abcd},
\end{align*}
which is automatic from the symmetry of $C \in
\frakg \otimes \frakg$. Secondly, it must satisfy
\begin{align}
  \label{eq:16}
  \check R^\rho_{abcd} = - \check R^\rho_{bacd},
\end{align}
which holds, for instance, if the representation $\rho$ is
skew-adjoint with respect to $B_V$. Finally, the tensor $\check
R^{\rho}$ must satisfy the algebraic Bianchi identity,
\begin{align}
  \label{eq:17}
   \check R^\rho_{abcd} + \check R^\rho_{bcad} + \check R^\rho_{cabd} = 0.
\end{align}
For convenience, we record the properties required of $\rho$ in the
following definition.

\begin{definition}
  \label{def:7}
  A representation $\rho \colon \frakg \to \End(V)$ of a metrized Lie
  algebra $\frakg$, with Casimir tensor $C \in \frakg \otimes
  \frakg$, on a vector space $V$ with a non-degenerate bilinear form
  $B$, is said to have \emph{curvature symmetries} if the tensor
  $\check R^{\rho} \in (V^*)^{\otimes 4}$ given by
  \begin{align*}
    \check R^{\rho}_{abcd} = \rho(C)^{xy}_{ac}B_{xb}B_{yd}
  \end{align*}
  satisfies the skew-symmetry in \eqref{eq:16} and the Bianchi identity
  in \eqref{eq:17}.
\end{definition}

These necessary symmetries are in fact sufficient to realize a Lie
algebra weight tensor $\rho(C)$ on a symmetric space. Indeed,
we can use the tensors $B_V$ and $\check R^{\rho}$, as well as the
corresponding $R^\rho \in (V^* \wedge V^*) \otimes \End(V)$, with a raised
index, to construct a Lie algebra $\frakh = R^{\rho}(V \wedge V)$ and
a symmetric triple $(\frakh \oplus V, s, B)$, exactly as we did using
the metric and curvature on a single tangent space of a symmetric
space. To prove that the bracket on $\frakh \oplus V$ satisfies the
Jacobi identity, both the Bianchi identity and the 4T relation are
needed, but the latter is automatically satisfied for the tensor
$\rho(C)$, as noted in \eqref{eq:15}. The Lie algebra $\frakh$ acts on
$(V,B_V)$ by skew-symmetric endomorphisms, and the weight tensor of
this representation reproduces $\rho(C)$. In other words, the weight
tensor must be realizable through an orthogonal representation and
this representation must satisfy the Bianchi identity in the sense of
\Fref{def:7}.

By applying \Fref{thm:7}, we
have proved the following.

\begin{theorem}
  \label{thm:8}
  The weight tensor of a real Lie algebra weight system can be
  realized on a symmetric space if and only if it comes from
  a representation with curvature symmetries.
\end{theorem}

\begin{example}
  \label{exa:11}
  The weight tensor obtained from the standard representation of
  $\sl_2$ (\Fref{exa:3}) cannot be realized on a symmetric
  space. Indeed, by the discussions above, the corresponding holonomy
  algebra would by a subalgebra of the one-dimensional Lie algebra
  $\so_2$, which certainly does not realize the weight tensor
  \eqref{eq:14}.
\end{example}

\vspace{0.5cm}

\subsection*{Acknowledgements}

We would like to thank J{\o}rgen Ellegaard Andersen,
Daniel Tubbenhauer and Florian Sch\"atz for many helpful discussions.

\clearpage

%%% Change bibliography to references
\renewcommand{\bibname}{References}

\def\cprime{$'$} \def\cprime{$'$} \def\cprime{$'$}


\begin{thebibliography}{CDM}

\bibitem[Bal]{MR2243012}
W.~Ballmann.
\newblock {\em Lectures on {K}\"ahler manifolds}.
\newblock ESI Lectures in Mathematics and Physics. European Mathematical
  Society (EMS), Z\"urich, 2006.

\bibitem[BN]{MR1318886}
D.~Bar-Natan.
\newblock On the {V}assiliev knot invariants.
\newblock {\em Topology}, 34(2):423--472, 1995.

\bibitem[BB]{MR1854694}
A.~Beliakova and C.~Blanchet.
\newblock Modular categories of types {$B$}, {$C$} and {$D$}.
\newblock {\em Comment. Math. Helv.}, 76(3):467--500, 2001.

\bibitem[Bes]{MR867684}
A.~L. Besse.
\newblock {\em Einstein manifolds}, volume~10 of {\em Ergebnisse der Mathematik
  und ihrer Grenzgebiete (3) [Results in Mathematics and Related Areas (3)]}.
\newblock Springer-Verlag, Berlin, 1987.

\bibitem[BL]{MR1198809}
J.~S. Birman and X.-S. Lin.
\newblock Knot polynomials and {V}assiliev's invariants.
\newblock {\em Invent. Math.}, 111(2):225--270, 1993.

\bibitem[CP]{MR556610}
M.~Cahen and M.~Parker.
\newblock Pseudo-{R}iemannian symmetric spaces.
\newblock {\em Mem. Amer. Math. Soc.}, 24(229):iv+108, 1980.

\bibitem[Car]{MR1221650}
P.~Cartier.
\newblock Construction combinatoire des invariants de {V}assiliev-{K}ontsevich
  des n\oe uds.
\newblock {\em C. R. Acad. Sci. Paris S{\'e}r. I Math.}, 316(11):1205--1210,
  1993.

\bibitem[CDM]{MR2962302}
S.~Chmutov, S.~Duzhin, and J.~Mostovoy.
\newblock {\em Introduction to {V}assiliev knot invariants}.
\newblock Cambridge University Press, Cambridge, 2012.

\bibitem[Gor]{MR1738389}
V.~Goryunov.
\newblock Vassiliev invariants of knots in {$\mathbf{R}^3$} and in a solid
  torus.
\newblock In {\em Differential and symplectic topology of knots and curves},
  volume 190 of {\em Amer. Math. Soc. Transl. Ser. 2}, pages 37--59. Amer.
  Math. Soc., Providence, RI, 1999.

\bibitem[Hel]{MR514561}
S.~Helgason.
\newblock {\em Differential geometry, {L}ie groups, and symmetric spaces},
  volume~80 of {\em Pure and Applied Mathematics}.
\newblock Academic Press, Inc. [Harcourt Brace Jovanovich, Publishers], New
  York-London, 1978.

\bibitem[Kap]{MR1671737}
M.~Kapranov.
\newblock Rozansky-{W}itten invariants via {A}tiyah classes.
\newblock {\em Compositio Math.}, 115(1):71--113, 1999.

\bibitem[Kas]{MR1321145}
C.~Kassel.
\newblock {\em Quantum groups}, volume 155 of {\em Graduate Texts in
  Mathematics}.
\newblock Springer-Verlag, New York, 1995.

\bibitem[KO]{MR2577201}
I.~Kath and M.~Olbrich.
\newblock On the structure of pseudo-{R}iemannian symmetric spaces.
\newblock {\em Transform. Groups}, 14(4):847--885, 2009.

\bibitem[Kon1]{MR1237836}
M.~Kontsevich.
\newblock Vassiliev's knot invariants.
\newblock In {\em I. {M}. {G}el\cprime fand {S}eminar}, volume~16 of {\em Adv.
  Soviet Math.}, pages 137--150. Amer. Math. Soc., Providence, RI, 1993.

\bibitem[Kon2]{MR1671725}
M.~Kontsevich.
\newblock Rozansky-{W}itten invariants via formal geometry.
\newblock {\em Compositio Math.}, 115(1):115--127, 1999.

\bibitem[LM]{MR1394520}
T.~Q.~T. Le and J.~Murakami.
\newblock The universal {V}assiliev-{K}ontsevich invariant for framed oriented
  links.
\newblock {\em Compositio Math.}, 102(1):41--64, 1996.

\bibitem[Neu]{math/0301326}
T.~Neukirchner.
\newblock Solvable pseudo-{R}iemannian symmetric spaces, 2003.

\bibitem[Pet]{MR2243772}
P.~Petersen.
\newblock {\em Riemannian geometry}, volume 171 of {\em Graduate Texts in
  Mathematics}.
\newblock Springer, New York, second edition, 2006.

\bibitem[Piu1]{MR1324388}
S.~Piunikhin.
\newblock Combinatorial expression for universal {V}assiliev link invariant.
\newblock {\em Comm. Math. Phys.}, 168(1):1--22, 1995.

\bibitem[Piu2]{MR1321294}
S.~Piunikhin.
\newblock Weights of {F}eynman diagrams, link polynomials and {V}assiliev knot
  invariants.
\newblock {\em J. Knot Theory Ramifications}, 4(1):163--188, 1995.

\bibitem[RT]{MR1036112}
N.~Y. Reshetikhin and V.~G. Turaev.
\newblock Ribbon graphs and their invariants derived from quantum groups.
\newblock {\em Comm. Math. Phys.}, 127(1):1--26, 1990.

\bibitem[RW]{MR1481135}
L.~Rozansky and E.~Witten.
\newblock Hyper-{K}{\"a}hler geometry and invariants of three-manifolds.
\newblock {\em Selecta Math. (N.S.)}, 3(3):401--458, 1997.

\bibitem[Tur]{MR939474}
V.~G. Turaev.
\newblock The {Y}ang-{B}axter equation and invariants of links.
\newblock {\em Invent. Math.}, 92(3):527--553, 1988.

\bibitem[Yam]{MR1016274}
S.~Yamada.
\newblock An invariant of spatial graphs.
\newblock {\em J. Graph Theory}, 13(5):537--551, 1989.

\end{thebibliography}
\end{document}